\documentclass[a4paper, 1wpt]{amsart}
\usepackage{amsmath,amsthm,amssymb,amsfonts,amscd,color}
\usepackage[arrow,matrix]{xy}
\usepackage{graphics}
\usepackage[dvips]{graphicx}
\usepackage{tikz}

\theoremstyle{plain} \numberwithin{equation}{section}
\newtheorem{theorem}{Theorem}[section]

\newtheorem{proposition}[theorem]{Proposition}
\newtheorem{lemma}[theorem]{Lemma}

\theoremstyle{definition}
\newtheorem*{definition}{Definition}
\newtheorem{example}{Example}[section]
\newtheorem*{remark}{Remark}
\newtheorem*{con}{Construction}

\def\Z{\mathbb Z}
\def\C{\mathbb C}
\def\R{\mathbb R}
\def\Q{\mathbb Q}
\def\F{\mathcal F}

\DeclareMathOperator{\rank}{rank} 
\DeclareMathOperator{\Int}{Int}
\DeclareMathOperator{\codim}{codim}
\DeclareMathOperator{\Hom}{Hom}
\DeclareMathOperator{\coker}{coker}

 \DeclareMathOperator{\im}{im}
\DeclareMathOperator{\id}{id}

\newcommand{\tM}{\tilde M}

\newcommand{\PsiE}{\Psi_E}
\newcommand{\PsiV}{\Psi_V}
\newcommand{\Sing}{W}
\newcommand{\Dd}{\mathcal{D}}
\newcommand{\Ff}{\mathcal{F}}
\newcommand{\ta}{\mathfrak{t}}

\newcommand{\jh}{\widehat{j}}
\newcommand{\inc}{\kappa}

\def\P{\mathcal P}

\def\kr{\mathcal R}

\begin{document}
\title{Cohomology of toric origami manifolds with acyclic proper faces}

\author[A. Ayzenberg]{Anton Ayzenberg}
\address{Department of Mathematics, Osaka City University, Sumiyoshi-ku, Osaka 558-8585, Japan.}
\curraddr{Faculty of Mathematics, National Research University Higher School of Economics, Moscow 117312, Russia}
\email{ayzenberga@gmail.com}

\author[M. Masuda]{Mikiya Masuda}
\address{Department of Mathematics, Osaka City University, Sumiyoshi-ku, Osaka 558-8585, Japan.}
\email{masuda@sci.osaka-cu.ac.jp}

\author[S. Park]{Seonjeong Park}
\address{Division of Mathematical Models, National Institute for Mathematical Sciences, 463-1 Jeonmin-dong, Yuseong-gu, Daejeon 305-811, Korea}
\email{seonjeong1124@nims.re.kr}

\author[H. Zeng]{Haozhi Zeng}
\address{Department of Mathematics, Osaka City University, Sumiyoshi-ku, Osaka 558-8585, Japan.}
\email{zenghaozhi@icloud.com}


\date{\today}
\thanks{The first author is supported by the JSPS postdoctoral fellowship program. The second author was partially supported by Grant-in-Aid for Scientific Research 25400095}
\subjclass[2000]{Primary 57S15, 53D20; Secondary 14M25, 57R91,
55N91} \keywords{toric origami manifold, origami template, Delzant
polytope, symplectic toric manifold, equivariant cohomology, face
ring, simplicial poset}

\begin{abstract}
A toric origami manifold is a generalization of a symplectic toric
manifold (or a toric symplectic manifold). The origami symplectic form is allowed to degenerate in a good
controllable way in contrast to the
usual symplectic form. It is widely known that symplectic toric
manifolds are encoded by Delzant polytopes, and the cohomology and
equivariant cohomology rings of a symplectic toric manifold can be
described in terms of the corresponding polytope. Recently, Holm and Pires described the cohomology of a toric origami manifold
$M$ in terms of the orbit space $M/T$ when $M$ is orientable and
the orbit space $M/T$ is contractible. But in general the orbit space of a toric
origami manifold need not be contractible. In this paper we study
the topology of orientable toric origami manifolds for the wider
class of examples: we require that every proper face of the orbit
space is acyclic, while the orbit space itself may be arbitrary.
Furthermore, we give a general description of the equivariant
cohomology ring of torus manifolds with locally standard torus
actions in the case when proper faces of the orbit space are
acyclic and the free part of the action is a trivial torus bundle.
\end{abstract}

\maketitle

\section{Introduction}

A symplectic toric manifold is a compact
connected symplectic manifold of dimension $2n$ with an effective
Hamiltonian action of a compact $n$-dimensional torus $T$. A
famous result of Delzant \cite{de88} describes a bijective
correspondence between symplectic toric manifolds and simple
convex polytopes, called Delzant polytopes. The polytope
associated to a symplectic toric manifold $M$ is the image of the
moment map on $M$.

Origami manifolds appeared in differential geometry  recently as a
generalization of symplectic manifolds \cite{ca-gu-pi11}.
A folded symplectic form on a $2n$-dimensional manifold $M$ is a
closed $2$-form $\omega$ whose top power $\omega^n$ vanishes
transversally on a subset $\Sing$ and whose restriction to points
in $\Sing$ has maximal rank. Then $\Sing$ is a codimension-one
submanifold of $M$, called the fold. The maximality of the
restriction of $\omega$ to $\Sing$ implies the existence of a line
field on $\Sing$. If the
line field is the vertical bundle of some principal
$S^1$-fibration $\Sing\to X$, then $\omega$ is called an \emph{origami form}.

Toric origami manifolds are generalizations of symplectic toric manifolds.
The notions of a Hamiltonian action
and a moment map are defined similarly to the symplectic case, and
a toric origami manifold is defined to be a compact connected
origami manifold $(M^{2n}, \omega)$ equipped with an effective
Hamiltonian action of a torus $T$. Similarly to Delzant's theorem
for symplectic toric manifolds, toric origami manifolds
bijectively correspond to special combinatorial structures, called
origami templates, via moment maps \cite{ca-gu-pi11}. An origami
template is a collection of Delzant polytopes with some additional
gluing data encoded by a template graph $G$.

A problem of certain interest is to describe the cohomology ring and $T$-equivariant cohomology
ring of toric origami manifold $M$ in terms of the corresponding
origami template, see \cite{ca-gu-pi11} and \cite{ho-pi12}. In this paper
we present some partial results concerning this problem.

While in general toric origami manifolds can be non-orientable, in
this paper we restrict to the orientable case. Under this
assumption the action of $T$ on a toric origami manifold $M$ is
locally standard, so the orbit space $M/T$ is a manifold with
corners. One can describe the orbit space $M/T$ as a result of
gluing polytopes of the origami template. This shows that $M/T$ is
homotopy equivalent to the template graph $G$. Every proper face of
$M/T$ is homotopy equivalent to some subgraph of $G$. Thus a toric
origami manifold has the property that the orbit space and all its
faces are homotopy equivalent to wedges of circles or
contractible.

If $G$ is a tree, then $M/T$ and all its faces are contractible.
Hence, a general result of \cite{ma-pa06} applies. It gives a
description similar to toric varieties (or quasitoric manifolds):
$H^*_T(M)\cong \Z[M/T]$ and $H^*(M)\cong
\Z[M/T]/(\theta_1,\ldots,\theta_n)$. Here, $\Z[M/T]$ is the face ring
of the manifold with corners $M/T$, and
$(\theta_1,\ldots,\theta_n)$ is the ideal generated by the linear
system of parameters, defined by the characteristic map on $M/T$.
This case is discussed in detail in \cite{ho-pi12}. But, if $G$ has
cycles, even the Betti numbers of $M$ remain unknown in general. Only when $M$ is of dimension $4$, Holm and Pires described the Betti numbers of $M$ in \cite{ho-pi13}.

In this paper we study the cohomology of an orientable toric
origami manifold $M$ in the case when $M/T$ is itself arbitrary,
but every proper face of $M/T$ is acyclic (this assumption is not always satisfied, see Section~\ref{sectNonAcyclicFaces}). A different approach to
this task, based on the spectral sequence of the filtration by
orbit types, is proposed in a more general situation in
\cite{ayze14}. For toric origami manifolds, the calculation of
Betti numbers in this paper gives the same answer but a simpler
proof.

The paper is organized as follows. Section
\ref{sectToricOrigami} contains necessary definitions and
properties of toric origami manifolds and origami templates. In
Section \ref{sectBettiNumbers} we describe the procedure which
simplifies a given toric origami manifold step-by-step, and give
an inductive formula for Betti numbers. Section
\ref{sectBettiFace} provides more convenient formulas expressing
Betti numbers of $M$ in terms of the first Betti number of $M/T$
and the face numbers of the dual simplicial poset. Section
\ref{sectEquivar} is devoted to an equivariant cohomology. While
toric origami manifolds serve as a motivating example, we describe
the equivariant cohomology ring in a more general setting. In
Section \ref{sectSerreSpec}, we describe the properties of the
Serre spectral sequence of the fibration $\pi\colon ET\times_TM\to
BT$ for a toric origami manifold $M$. The restriction homomorphism
$\iota^*\colon H^*_T(M)\to H^*(M)$ induces a graded ring
homomorphism $\bar\iota^*\colon H^*_T(M)/(\pi^*(H^{2}(BT)))\to
H^*(M)$. In Section \ref{sectTowardsRing} we use Schenzel's
theorem and the calculations of previous sections to show that
$\bar\iota^*$ is an isomorphism except in degrees $2$, $4$ and
$2n-1$, a monomorphism in degrees $2$ and $2n-1$, and an epimorphism in
degree $4$; we also find the ranks of the kernels and cokernels in these exceptional degrees. Since $\bar\iota^*$ is a ring
homomorphism, these considerations describe the product structure
on the most part of $H^*(M)$, except for the cokernel of
$\bar\iota^*$ in degree $2$. Section \ref{sect4dimCase}
illustrates our considerations in the $4$-dimensional case. In
section \ref{sectCokernel} we give a geometrical description of
the cokernel of $\bar\iota^*$ in degree 2, and suggest a partial
description of the cohomology multiplication for these extra
elements. The discussion of Section~\ref{sectNonAcyclicFaces}
shows which part of the results can be generalized to the case of
non-acyclic faces.

\section{Toric origami manifolds}\label{sectToricOrigami}

In this section, we recall the definitions and properties of toric
origami manifolds and origami templates. Details can be found in
\cite{ca-gu-pi11}, \cite{ma-pa13} or \cite{ho-pi12}.

A folded symplectic form on a $2n$-dimensional manifold $M$ is a
closed 2-form $\omega$ whose top power $\omega^n$ vanishes
transversally on a subset $\Sing$ and whose restriction to points
in $\Sing$ has maximal rank. Then $\Sing$ is a codimension-one
submanifold of $M$ and is called the fold. If $\Sing$ is empty,
$\omega$ is a genuine symplectic form. The pair $(M, \omega)$ is
called a folded symplectic manifold. Since the restriction of
$\omega$ to $\Sing$ has maximal rank, it has a one-dimensional
kernel at each point of $\Sing$. This determines a line field on
$\Sing$ called the null foliation. If the null foliation is the
vertical bundle of some principal $S^1$-fibration $\Sing\to X$
over a compact base $X$, then the folded symplectic form $\omega$
is called an origami form and the pair $(M,\omega)$ is called an
origami manifold. The action of a torus $T$ on an origami manifold
$(M,\omega)$ is Hamiltonian if it admits a moment map $\mu\colon
M\to \ta^*$ to the dual Lie algebra of the torus, which satisfies
the conditions: (1) $\mu$ is equivariant with respect to the given
action of $T$ on $M$ and the coadjoint action of $T$ on the vector
space $\ta^*$ (this action is trivial for the torus); (2) $\mu$
collects Hamiltonian functions, that is, $d\langle\mu,V\rangle =
\imath_{V^\#}\omega$ for each $V\in \ta$, where $V^\#$ is the
vector field on $M$ generated by $V$.

\begin{definition}
A toric origami manifold $(M,\omega,T,\mu)$, abbreviated as $M$,
is a compact connected origami manifold $(M,\omega)$ equipped with
an effective Hamiltonian action of a torus $T$ with $\dim T =
\frac12\dim M$ and with a choice of a corresponding moment map
$\mu$.
\end{definition}

When the fold $\Sing$ is empty, a toric origami manifold is a
symplectic toric manifold. A theorem of Delzant \cite{de88} says
that symplectic toric manifolds are classified by their moment
images called Delzant polytopes. Recall that a Delzant polytope in
$\R^n$ is a simple convex polytope, whose normal fan is smooth
(with respect to some given lattice $\Z^n\subset\R^n$). In other
words, all normal vectors to the facets of $P$ have rational
coordinates, and the primitive normal vectors
$\nu(F_1),\ldots,\nu(F_n)$ form a basis of the lattice $\Z^n$ whenever facets $F_1,\ldots,F_n$ meet in a
vertex of $P$. Let
$\Dd_n$ denote the set of all Delzant polytopes in $\R^n$ (w.r.t.
a given lattice) and $\Ff_n$ be the set of all their facets.

The moment data of a toric origami manifold can be encoded into an
origami template $(G,\PsiV,\PsiE)$, where
\begin{itemize}
\item $G$ is a connected graph (loops and multiple edges are allowed) with the
vertex set $V$ and edge set $E$;
\item $\PsiV\colon V\to \Dd_n$;
\item $\PsiE\colon E\to \Ff_n$;
\end{itemize}
subject to the following conditions:
\begin{itemize}
\item If $e\in E$ is an edge of $G$ with endpoints $v_1, v_2\in
V$, then $\PsiE(e)$ is a facet of both polytopes $\PsiV(v_1)$ and
$\PsiV(v_2)$, and these polytopes coincide near $\PsiE(e)$ (this
means there exists an open neighborhood $U$ of $\PsiE(e)$ in
$\R^n$ such that $U\cap\PsiV(v_1)=U\cap\PsiV(v_2)$).
\item If $e_1,e_2\in E$ are two edges of $G$ adjacent to $v\in V$,
then $\PsiE(e_1)$ and $\PsiE(e_2)$ are disjoint facets of
$\Psi(v)$.
\end{itemize}

The facets of the form $\PsiE(e)$ for $e\in E$ are called the fold
facets of the origami template.

The following is a generalization of the theorem by Delzant to
toric origami manifolds.

\begin{theorem}[\cite{ca-gu-pi11}]\label{theo:classifOrigami}
Assigning the moment data of a toric origami manifold induces a
one-to-one correspondence
\[
\{\mbox{toric origami
manifolds}\}\leftrightsquigarrow\{\mbox{origami templates}\}
\]
up to equivariant origami symplectomorphism on the left-hand side,
and affine equivalence on the right-hand side.
\end{theorem}

Denote by $|(G,\PsiV,\PsiE)|$ the topological space constructed
from the disjoint union $\bigsqcup_{v\in V}\PsiV(v)$ by
identifying facets $\PsiE(e)\subset\PsiV(v_1)$ and
$\PsiE(e)\subset\PsiV(v_2)$ for every edge $e\in E$ with endpoints
$v_1$ and $v_2$.

An origami template $(G,\PsiV,\PsiE)$ is called co\"{o}rientable
if the graph $G$ has no loops (this means all edges have different
endpoints). Then the corresponding toric origami manifold is also
called co\"{o}rientable. If $M$ is orientable, then $M$ is
co\"{o}rientable \cite{ho-pi12}. If $M$ is co\"{o}rientable, then
the action of $T^n$ on $M$ is locally standard \cite[Lemma
5.1]{ho-pi12}. We review the definition of locally standard action
in Section \ref{sectEquivar}.

Let $(G,\PsiV,\PsiE)$ be an origami template and $M$ the
associated toric origami manifold which is supposed to be
orientable in the following. The topological space
$|(G,\PsiV,\PsiE)|$ is a manifold with corners with the face
structure induced from the face structures on polytopes
$\PsiV(v)$, and the space $|(G,\PsiV,\PsiE)|$ is homeomorphic to $M/T$ as a
manifold with corners. The space $|(G,\PsiV,\PsiE)|$ has the same
homotopy type as the graph $G$, thus $M/T\cong |(G,\PsiV,\PsiE)|$
is either contractible or homotopy equivalent to a wedge of
circles.

Under the correspondence of Theorem \ref{theo:classifOrigami}, the
fold facets of the origami template correspond to the connected
components of the fold $\Sing$ of $M$. If $F=\PsiE(e)$ is a fold
facet of the template $(G,\PsiV,\PsiE)$, then the corresponding
component $Z\subseteq\mu^{-1}(F)$ of the fold $\Sing\subset M$ is a
principal $S^1$-bundle over a compact space $B$. The space $B$ is
a $(2n-2)$-dimensional symplectic toric manifold corresponding to
the Delzant polytope $F$. In the following we also call the
connected components $Z$ of the fold $\Sing$ the ``folds'' by
abuse of terminology.

\section{Betti numbers of toric origami manifolds} \label{sectBettiNumbers}

Let $M$ be an orientable toric origami manifold of dimension $2n$
with a fold $Z$. Let $F$ be the corresponding folded facet in the
origami template of $M$ and let $B$ be the symplectic toric
manifold corresponding to $F$.  The normal line bundle of $Z$ to
$M$ is trivial so that an invariant closed tubular neighborhood of
$Z$ in $M$ can be identified with $Z\times [-1,1]$.  We set
\[
\tM:=M-\Int(Z\times [-1,1]).
\]
This has two boundary components which are copies of $Z$.  We
close $\tM$ by gluing two copies of the disk bundle associated to
the principal $S^1$-bundle $Z\to B$ along their boundaries.  The
resulting closed manifold (possibly disconnected), denoted $M'$,
is again a toric origami manifold and the graph associated to $M'$ is the graph associated to $M$ with the edge corresponding to the folded facet $F$ removed.

Let $G$ be the graph associated to the origami template of $M$ and
let $b_1(G)$ be its first Betti number. We assume that $b_1(G)\ge
1$. A folded facet in the origami template of $M$ corresponds to
an edge of $G$. We choose an edge $e$ in a (non-trivial) cycle of
$G$ and let $F$, $Z$ and $B$ be respectively the folded facet, the
fold and the symplectic toric manifold corresponding to the edge
$e$. Then $M'$ is connected and since the graph $G'$ associated to $M'$
is the graph $G$ with the edge $e$ removed, we have
$b_1(G')=b_1(G)-1$.

Two copies of $B$ lie in $M'$ as closed submanifolds, denoted
$B_+$ and $B_-$.  Let $N_+$ (resp. $N_-$) be an invariant closed
tubular neighborhood of $B_+$ (resp. $B_-$) and $Z_+$ (resp.
$Z_-$) be the boundary of $N_+$ (resp. $N_-$). Note that
$M'-\Int(N_+\cup N_-)$ can naturally be identified with $\tM$, so
that
\[
\tM=M'-\Int(N_+\cup N_-)=M-\Int(Z\times [-1,1])
\]
and
\begin{alignat}{2}
M'&=\tM\cup(N_+\cup N_-),\hspace{1cm} &\tM\cap(N_+\cup N_-)=Z_+\cup Z_-,\label{eq:3-1-1} \\
M&=\tM\cup(Z\times[-1,1]),\hspace{1cm} &\tM\cap(Z\times[-1,1])=Z_+\cup Z_-. \label{eq:3-1-2}
\end{alignat}

\begin{remark} It follows from \eqref{eq:3-1-1} and \eqref{eq:3-1-2} that
\[
\chi(M')=\chi(\tM)+2\chi(B),\quad \chi(M)=\chi(\tM)
\]
and hence $\chi(M')=\chi(M)+2\chi(B)$.  Note that this formula
holds without the acyclicity assumption (made later) on proper
faces of $M/T$.
\end{remark}

We shall investigate relations among the Betti numbers of $M, M',
\tM, Z$, and $B$.  The spaces $\tM$ and $Z$ are auxiliary ones and
our aim is to find relations among the Betti numbers of $M, M'$
and $B$. In the following, all cohomology groups and Betti numbers
are taken with $\Z$ coefficients unless otherwise stated but the
reader will find that the same argument works over any field.

\begin{lemma} \label{lemm:3-1}
The Betti numbers of $Z$ and $B$ have the relation
$$b_{2i}(Z)-b_{2i-1}(Z)=b_{2i}(B)-b_{2i-2}(B)$$ for every~$i$.
\end{lemma}

\begin{proof}
Since $\pi\colon Z\to B$ is a principal $S^1$-bundle and
$H^{odd}(B)=0$, the Gysin exact sequence for the principal
$S^1$-bundle splits into a short exact sequence
\begin{equation} \label{eq:3-4}
0\to H^{2i-1}(Z)\to H^{2i-2}(B)\to H^{2i}(B)\xrightarrow{\pi^*} H^{2i}(Z)\to 0\quad \text{for every $i$}
\end{equation}
and this implies the lemma.
\end{proof}

\begin{lemma} \label{lemm:3-4}
The Betti numbers of $\tilde{M}$, $M'$, and $B$ have the relation
$$b_{2i}(\tM)-b_{2i-1}(\tM)=b_{2i}(M')-b_{2i-1}(M')-2b_{2i-2}(B)$$ for every $i$.
\end{lemma}

\begin{proof}
We consider the Mayer-Vietoris exact sequence in cohomology for
the triple $(M', \tM, N_+\cup N_-)$:
\begin{alignat*}{5}
\to &H^{2i-2}(M')&\to &H^{2i-2}(\tM)\oplus H^{2i-2}(N_+\cup N_-)&\to &H^{2i-2}(Z_+\cup Z_-)\\
\xrightarrow{\delta^{2i-2}}&H^{2i-1}(M')&\to &H^{2i-1}(\tM)\oplus H^{2i-1}(N_+\cup N_-)&\to &H^{2i-1}(Z_+\cup Z_-)\\
\xrightarrow{\delta^{2i-1}}&H^{2i}(M')&\to &H^{2i}(\tM)\oplus H^{2i}(N_+\cup N_-)&\to &H^{2i}(Z_+\cup Z_-)\\
\xrightarrow{\delta^{2i}}&H^{2i+1}(M')&\to& &&
\end{alignat*}
Since the inclusions $B=B_\pm\mapsto N_\pm$ are homotopy
equivalences and $Z_\pm=Z$, the restriction homomorphism
$H^q(N_+\cup N_-)\to H^q(Z_+\cup Z_-)$ above can be replaced by
$\pi^*\oplus\pi^*\colon H^q(B)\oplus H^q(B)\to H^q(Z)\oplus
H^q(Z)$ which is surjective for even $q$ from the sequence~\eqref{eq:3-4}.
Therefore, $\delta^{2i-2}$ and $\delta^{2i}$ in the exact sequence
above are trivial. It follows that
\[
\begin{split}
&b_{2i-1}(M')-b_{2i-1}(\tM)-2b_{2i-1}(B)+2b_{2i-1}(Z)\\
-&b_{2i}(M')+b_{2i}(\tM)+2b_{2i}(B)-2b_{2i}(Z)=0.
\end{split}
\]
Here $b_{2i-1}(B)=0$ because $B$ is a symplectic toric manifold,
and $2b_{2i-1}(Z)+2b_{2i}(B)-2b_{2i}(Z)=2b_{2i-2}(B)$ by
Lemma~\ref{lemm:3-1}.  Using these identities, the identity above
reduces to the identity in the lemma.
\end{proof}

Next we consider the Mayer-Vietoris exact sequence in cohomology
for the triple $(M,\tM,Z\times [-1,1])$:
\begin{alignat*}{5}
\to &H^{2i-2}(M)&\to &H^{2i-2}(\tM)\oplus H^{2i-2}(Z\times[-1,1])&\to &H^{2i-2}(Z_+\cup Z_-)\\
\to &H^{2i-1}(M)&\to &H^{2i-1}(\tM)\oplus H^{2i-1}(Z\times[-1,1])&\to &H^{2i-1}(Z_+\cup Z_-)\\
\to &H^{2i}(M)&\to &H^{2i}(\tM)\oplus H^{2i}(Z\times[-1,1])&\to &H^{2i}(Z_+\cup Z_-)\to
\end{alignat*}
We make the following assumption:
\begin{quote}
$(*)$ The restriction map $H^{2j}(\tM)\oplus H^{2j}(Z\times[-1,1])\to H^{2j}(Z_+\cup Z_-)$ in the Mayer-Vietoris sequence above is surjective for $j\ge 1$.
\end{quote}
Note that the restriction map above is not surjective when $j=0$ because the image is the diagonal copy of $H^0(Z)$ in this case and we will see in Lemma~\ref{lemm:3-5} below that the assumption
$(*)$ is satisfied when every proper face of $M/T$ is acyclic.

\begin{lemma} \label{lemm:3-6}
Suppose that the assumption $(*)$ is satisfied.  Then
\[
\begin{split}
&b_2(\tM)-b_1(\tM)=b_2(M)-b_1(M)+b_2(B),\\
&b_{2i}(\tM)-b_{2i-1}(\tM)=b_{2i}(M)-b_{2i-1}(M)+b_{2i}(B)-b_{2i-2}(B) \quad\text{for $i\ge 2$}.
\end{split}
\]
\end{lemma}

\begin{proof}
By the assumption $(*)$, the Mayer-Vietoris exact sequence for the
triple $(M,\tM,Z\times [-1,1])$ splits into short exact sequences:
\begin{alignat*}{5}
0\to &H^{0}(M)&\to &H^{0}(\tM)\oplus H^{0}(Z\times[-1,1])&\to &H^{0}(Z_+\cup Z_-)\\
 \to &H^{1}(M)&\to &H^{1}(\tM)\oplus H^{1}(Z\times[-1,1])&\to &H^{1}(Z_+\cup Z_-)\\
\to &H^{2}(M)&\to &H^{2}(\tM)\oplus H^{2}(Z\times[-1,1])&\to &H^{2}(Z_+\cup Z_-)\to 0
\end{alignat*}
and for $i\ge 2$
\begin{alignat*}{5}
0 \to &H^{2i-1}(M)&\to &H^{2i-1}(\tM)\oplus H^{2i-1}(Z\times[-1,1])&\to &H^{2i-1}(Z_+\cup Z_-)\\
\to &H^{2i}(M)&\to &H^{2i}(\tM)\oplus H^{2i}(Z\times[-1,1])&\to &H^{2i}(Z_+\cup Z_-)\to 0.
\end{alignat*}
The former short exact sequence above yields
\[
b_2(\tM)-b_1(\tM)=b_2(M)-b_1(M)+b_2(Z)-b_1(Z)+1
\]
while the latter above yields
\[
b_{2i}(\tM)-b_{2i-1}(\tM)=b_{2i}(M)-b_{2i-1}(M)+b_{2i}(Z)-b_{2i-1}(Z) \quad\text{for $i\ge 2$}.
\]
Here $b_{2i}(Z)-b_{2i-1}(Z)=b_{2i}(B)-b_{2i-2}(B)$ for every $i$ by Lemma~\ref{lemm:3-1}, so our lemma follows.
\end{proof}

\begin{lemma} \label{lemm:3-7}
Suppose that the assumption $(*)$ is satisfied and $n\ge 2$.  Then
\[
\begin{split}
&b_1(M')=b_1(M)-1,\quad b_2(M')=b_2(M)+b_2(B)+1,\\
&b_{2i+1}(M')=b_{2i+1}(M)\quad \text{for $1\le i\le n-2$}.
\end{split}
\]
\end{lemma}

\begin{proof}
It follows from Lemma~\ref{lemm:3-4} and Lemma~\ref{lemm:3-6} that
\begin{equation}\label{eq:3-5}
b_{2i}(M')-b_{2i-1}(M')=b_{2i}(M)-b_{2i-1}(M)+b_{2i}(B)+b_{2i-2}(B) \quad\text{for $i\ge 2$}.
\end{equation}
Take $i=n$ in \eqref{eq:3-5} and use Poincar\'e duality.  Then we
obtain
\[
b_0(M')-b_1(M')=b_0(M)-b_1(M)+b_0(B)
\]
which reduces to the first identity in the lemma. This together
with the first identity in Lemma~\ref{lemm:3-6} implies the second
identity in the lemma.

Similarly, take $i=n-1(\ge 2)$ in \eqref{eq:3-5} and use
Poincar\'e duality.  Then we obtain
\[
b_2(M')-b_3(M')=b_2(M)-b_3(M)+b_0(B)+b_2(B).
\]
This together with the second identity in the lemma implies
$b_3(M')=b_3(M)$.

Take $i$ to be $n-i$ in \eqref{eq:3-5} (so $2\le i\le n-2$) and
use Poincar\'e duality.  Then we obtain
\begin{equation*}
b_{2i}(M')-b_{2i+1}(M')=b_{2i}(M)-b_{2i+1}(M)+b_{2i-2}(B)+b_{2i}(B).
\end{equation*}
This together with \eqref{eq:3-5} implies
\[
b_{2i+1}(M')-b_{2i-1}(M')=b_{2i+1}(M)-b_{2i-1}(M) \quad \text{for $2\le i\le n-2$}.
\]
Since we know $b_3(M')=b_3(M)$, this implies the last identity in
the lemma.
\end{proof}

The following is a key lemma.

\begin{lemma} \label{lemm:3-5}
Suppose that every proper face of $M/T$ is acyclic.  Then the
homomorphism $H^{2j}(\tM)\to H^{2j}(Z_+\cup Z_-)$ induced from the
inclusion is surjective for $j\ge 1$, in particular, the
assumption $(*)$ is satisfied.
\end{lemma}

\begin{proof}
Since $B_+\cup B_-$ is a deformation retract of $N_+\cup N_-$, the
following diagram is commutative:
\[
\begin{CD}
H^{2j}(M')@>>> H^{2j}(B_+\cup B_-)\\
@VVV  @VV \pi_\pm^* V\\
H^{2j}(\tM) @>>> H^{2j}(Z_+\cup Z_-)
\end{CD}
\]
where $\pi_\pm\colon Z_+\cup Z_-\to B_+\cup B_-$ is the projection
and the other homomorphisms are induced from the inclusions.  By
\eqref{eq:3-4}, $\pi_\pm^*$ is surjective, so it suffices to show
that the homomorphism $H^{2j}(M')\to H^{2j}(B_+\cup B_-)$ is
surjective for $j\ge 1$.

The inverse image of a codimension $j$ face of $M'/T$ by the
quotient map $M'\to M'/T$ is a codimension $2j$ closed orientable
submanifold of $M'$ and defines an element of $H_{2n-2j}(M')$ so
that its Poincar\'e dual yields an element of $H^{2j}(M')$.  The
same is true for $B=B_+$ or $B_-$.  Note that $H^{2j}(B)$ is
additively generated by $\tau_{K}$'s where $K$ runs over all
codimension $j$ faces of $F=B/T$.

Set $F_\pm=B_\pm/T$, which are copies of the folded facet $F=B/T$.
Let $K_+$ be a codimension $j$ face of $F_+$.  Then there is a
codimension $j$ face $L$ of $M'/T$ such that $K_+=L\cap F_+$. We
note that $L\cap F_-=\emptyset$.  Indeed, if $L\cap
F_-\not=\emptyset$, then $L\cap F_-$ must be a codimension $j$
face of $F_-$, say $H_-$.  If $H_-$ is the copy $K_-$ of $K_+$,
then $L$ will create a codimension $j$ non-acyclic face of $M/T$
which contradicts the acyclicity assumption on proper faces of
$M/T$.  Therefore, $H_-\not=K_-$.  However, $F_\pm$ are
respectively facets of some Delzant polytopes, say $P_\pm$, and
the neighborhood of $F_+$ in $P_+$ is same as that of $F_-$ in
$P_-$ by definition of an origami template (although $P_+$ and
$P_-$ may not be isomorphic).  Let $\bar H$ and $\bar K$ be the
codimension $j$ faces of $P_-$ such that $\bar H\cap F=H_-$ and
$\bar K\cap F=K_-$.  Since $H_-\not=K_-$, the normal cones of
$\bar H$ and $\bar K$ are different.  However, these normal cones
must agree with that of $L$ because $L\cap F_+=K_+$ and $L\cap
F_-=H_-$ and the neighborhood of $F_+$ in $P_+$ is same as that of
$F_-$ in $P_-$.  This is a contradiction.

The codimension $j$ face $L$ of $M'/T$ associates an element
$\tau_L\in H^{2j}(M')$.  Since $L\cap F_+=K_+$ and $L\cap
F_-=\emptyset$, the restriction of $\tau_L$ to $H^{2j}(B_+\cup
B_-)=H^{2j}(B_+)\oplus H^{2j}(B_-)$ is $(\tau_{K_+},0)$, where
$\tau_{K_+}\in H^{2j}(B_+)$ is associated to $K_+$.  Since
$H^{2j}(B_+)$ is additively generated by $\tau_{K_+}$'s where
$K_+$ runs over all codimension $j$ faces of $F_+$, for each element $(x_+,0)$ in $H^{2j}(B_+)\oplus
H^{2j}(B_-)=H^{2j}(B\cup B_-)$, there is an element $y_+\in
H^{2j}(M')$ whose restriction image is $(x_+,0)$.  The same is
true for each element $(0,x_-)\in H^{2j}(B_+)\oplus H^{2j}(B_-)$.
This implies the lemma.
\end{proof}

Finally, we obtain the following.

\begin{theorem} \label{theo:3-1}
Let $M$ be an orientable toric origami manifold of dimension $2n$
$(n\ge 2)$ such that every proper face of $M/T$ is acyclic.  Then
\begin{equation} \label{eq:theo3-1}
b_{2i+1}(M)=0\quad\text{for $1\le i\le n-2$}.
\end{equation}
Moreover, if $M'$ and $B$ are as above, then
\begin{equation} \label{eq:theo3-1-1}
\begin{split}
&b_1(M')=b_1(M)-1\,\,(\text{hence $b_{2n-1}(M')=b_{2n-1}(M)-1$ by Poincar\'e duality}),\\
&b_{2i}(M')=b_{2i}(M)+b_{2i}(B)+b_{2i-2}(B)\quad\text{for $1\le i\le n-1$}.
\end{split}
\end{equation}
Finally, $H^*(M)$ is torsion free.
\end{theorem}

\begin{proof}
We have $b_1(M')=b_1(M)-1$ by Lemma~\ref{lemm:3-7}. Therefore, if
$b_1(M)=1$, then $b_1(M')=0$, that is, the graph associated to
$M'$ is acyclic and hence $b_{odd}(M')=0$ by \cite{ho-pi12} (or
\cite{ma-pa06}).  This together with Lemma~\ref{lemm:3-7} shows
that $b_{2i+1}(M)=0$ for $1\le i\le n-2$ when $b_1(M)=1$.  If
$b_1(M)=2$, then $b_1(M')=1$ so that $b_{2i+1}(M')=0$ for $1\le
i\le n-2$ by the observation just made and hence $b_{2i+1}(M)=0$
for $1\le i\le n-2$ by Lemma~\ref{lemm:3-7}.  Repeating this
argument, we see \eqref{eq:theo3-1}.

The relations in~\eqref{eq:theo3-1-1} follow from Lemma~\ref{lemm:3-7} and
\eqref{eq:3-5} together with the fact $b_{2i+1}(M)=0$ for $1\le
i\le n-2$.

As we remarked before Lemma~\ref{lemm:3-1}, the arguments developed in this section work with any field coefficients, in particular with $\Z/p$-coefficients
for every prime $p$. Hence \eqref{eq:theo3-1} and \eqref{eq:theo3-1-1}
hold for Betti numbers with $\Z/p$-coefficients. Accordingly, the Betti
numbers of $M$ with $\Z$-coefficients agree with the Betti numbers
of $M$ with $\Z/p$-coefficients for every prime $p$.  This implies
that $H^*(M)$ has no torsion.
\end{proof}


As for $H^1(M)$, we have a clear geometrical picture.

\begin{proposition} \label{prop:3-1}
Let $M$ be an orientable toric origami manifold of dimension $2n$
$(n\ge 2)$ such that every proper face of $M/T$ is acyclic. Let
$Z_1,\dots,Z_{b_1}$ be folds in $M$ such that the graph associated
to the origami template of $M$ with the $b_1$ edges corresponding
to $Z_1,\dots,Z_{b_1}$ removed is a tree.  Then
$Z_1,\dots,Z_{b_1}$ freely generate $H_{2n-1}(M)$, equivalently,
their Poincar\'e duals $z_1$, \dots, $z_{b_1}$ freely generate
$H^1(M)$.  Furthermore, all the products generated by
$z_1,\dots,z_{b_1}$ are trivial because $Z_1,\dots,Z_{b_1}$ are
disjoint and the normal bundle of $Z_j$ is trivial for each $j$.
\end{proposition}

\begin{proof}
We will prove the proposition by induction on $b_1$. When $b_1=0$, the proposition is trivial; so we may assume $b_1\ge 1$. Let $Z$ and $M'$
be as before.  Since $b_1(M')=b_1-1$, there are folds $Z_1, \dots,
Z_{b_1-1}$ in $M'$ such that $Z_1,\dots,Z_{b_1-1}$ freely generate
$H_{2n-1}(M')$ by induction assumption.  The folds
$Z_1,\dots,Z_{b_1-1}$ are naturally embedded in $M$ and we will
prove that these folds together with $Z$ freely generate
$H_{2n-1}(M)$.

We consider the Mayer-Vietoris exact sequence for the triple
$(M,\tM,Z\times [-1,1])$:
\[
\begin{split}
0&\to H_{2n}(M)\xrightarrow{\partial_*} H_{2n-1}(Z_+\!\cup\! Z_-) \xrightarrow{{\iota_1}_*\oplus {\iota_2}_*} H_{2n-1}(\tM)\oplus H_{2n-1}(Z\!\times\! [-1,1])\\
  &\to H_{2n-1}(M)\xrightarrow{\partial_*} H_{2n-2}(Z_+\!\cup\! Z_-) \xrightarrow{{\iota_1}_*\oplus {\iota_2}_*} H_{2n-2}(\tM)\oplus H_{2n-2}(Z\!\times\! [-1,1])
\end{split}
\]
where $\iota_1$ and $\iota_2$ are the inclusions. Since
$\iota_1^*\colon H^{2n-2}(\tM)\to H^{2n-2}(Z_+\cup Z_-)$ is
surjective by Lemma~\ref{lemm:3-5}, ${\iota_1}_*\colon
H_{2n-2}(Z_+\cup Z_-)\to H_{2n-2}(\tM)$ is injective when tensored
with $\Q$. However, $H^*(Z)$ has no torsion in odd degrees because
$H^{2i-1}(Z)$ is a subgroup of $H^{2i-2}(B)$ for every $i$ by
\eqref{eq:3-4} and $H^*(B)$ is torsion free.  Therefore, $H_*(Z)$
has no torsion in even degrees.  Therefore, ${\iota_1}_*\colon
H_{2n-2}(Z_+\cup Z_-)\to H_{2n-2}(\tM)$ is injective without
tensoring with $\Q$ and hence the above exact sequence reduces to
this short exact sequence:
\[
\begin{split}
0&\to H_{2n}(M)\xrightarrow{\partial_*} H_{2n-1}(Z_+\cup Z_-) \xrightarrow{{\iota_1}_*\oplus {\iota_2}_*} H_{2n-1}(\tM)\oplus H_{2n-1}(Z\times [-1,1])\\
  &\to H_{2n-1}(M)\to 0.
  \end{split}
\]
Noting $\partial_*([M])=[Z_+]-[Z_-]$ and
${\iota_2}_*([Z_\pm])=[Z]$, one sees that the above short exact
sequence implies an isomorphism
\begin{equation} \label{eq:3-6}
\iota_*\colon H_{2n-1}(\tM)\cong H_{2n-1}(M)
\end{equation}
where $\iota\colon \tM\to M$ is the inclusion map.

Consider the Mayer-Vietoris exact sequence for
$(M',\tM, N_+\cup N_-)$:
\[
\begin{split}
0&\to H_{2n}(M')\xrightarrow{\partial'_*} H_{2n-1}(Z_+\cup Z_-) \xrightarrow{{\iota_1}_*\oplus {\iota_3}_*} H_{2n-1}(\tM)\oplus H_{2n-1}(N_+\cup N_-)\\
  &\to H_{2n-1}(M')\xrightarrow{\partial'_*} H_{2n-2}(Z_+\cup Z_-) \xrightarrow{{\iota_1}_*\oplus {\iota_3}_*} H_{2n-2}(\tM)\oplus H_{2n-2}(N_+\cup N_-)
\end{split}
\]
where $\iota_3$ is the inclusion map of the unit sphere bundle in $N_+\cup N_-$. Note that $H_{2n-1}(N_+\cup
N_-)=H_{2n-1}(B_+\cup B_-)=0$ and ${\iota_1}_*\colon
H_{2n-2}(Z_+\cup Z_-)\to H_{2n-2}(\tM)$ is injective as observed
above.  Therefore, the above exact sequence reduces to this short
exact sequence:
\[
0\to H_{2n}(M')\xrightarrow{\partial'_*} H_{2n-1}(Z_+\cup Z_-) \xrightarrow{{\iota_1}_*} H_{2n-1}(\tM)
  \xrightarrow{\iota_*} H_{2n-1}(M')\to 0.
\]
Here $\partial_*([M])=[Z_+]-[Z_-]$ and $H_{2n-1}(M')$ is freely
generated by $Z_1$, \dots, $Z_{b_1-1}$ by induction assumption.
Therefore, the above short exact sequence implies that
\begin{equation*}
\text{$H_{2n-1}(\tM)$ is freely generated by $Z_1,\dots,Z_{b_1-1}$ and $Z_+$ (or $Z_-$).}
\end{equation*}
This together with \eqref{eq:3-6} completes the induction step and
proves the lemma.
\end{proof}

\section{Relations between Betti numbers and face numbers} \label{sectBettiFace}

Let $M$ be an orientable toric origami manifold of dimension $2n$
$(n\ge 2)$ such that every proper face of $M/T$ is acyclic.  In
this section we describe $b_{2i}(M)$ in terms of the face
numbers of $M/T$ and $b_1(M)$.  Let $\P$ be the simplicial poset
dual to $\partial(M/T)$.  As usual, we define
\[
\begin{split}
f_i&=\text{ the number of $(n-1-i)$-faces of $M/T$}\\
&=\text{ the number of $i$-simplices in $\P$}\quad \text{for $i=0,1,\dots,n-1$}
\end{split}
\]
and the $h$-vector $(h_0,h_1,\dots,h_n)$ by
\begin{equation} \label{eq:4-1}
\sum_{i=0}^nh_it^{n-i}=(t-1)^n+\sum_{i=0}^{n-1}f_i(t-1)^{n-1-i}.
\end{equation}

\begin{theorem} \label{theo:4-1}
Let $M$ be an orientable toric origami manifold of dimension $2n$
such that every proper face of $M/T$ is acyclic. Let $b_j$ be the
$j$th Betti number of $M$ and $(h_0,h_1,\dots,h_n)$ be the
$h$-vector of $M/T$.  Then
\[
\sum_{i=0}^nb_{2i}t^i=\sum_{i=0}^nh_it^i+b_1(1+t^n-(1-t)^n),
\]
in other words, $b_0=h_0=1$ and
\[
\begin{split}
b_{2i}&=h_i-(-1)^i\binom{n}{i}b_1\quad\text{for $1\le i\le n-1$},\\
b_{2n}&=h_n+(1-(-1)^n)b_1.
\end{split}
\]
\end{theorem}

\begin{remark}
Since every
proper face of $M/T$ is acyclic, we have $h_n=(-1)^n+\sum_{i=0}^{n-1}(-1)^{n-1-i}f_i$ by \eqref{eq:4-1}
and $\chi(\partial(M/T))=\sum_{i=0}^{n-1}(-1)^if_i$.  Therefore,
$h_n=(-1)^n-(-1)^n\chi(\partial(M/T))$.  Since $b_{2n}=1$, it
follows from the last identity in Theorem~\ref{theo:4-1} that
\[
\chi(\partial(M/T))-\chi(S^{n-1})=((-1)^n-1)b_1.
\]
Moreover, since $b_{2i}=b_{2n-2i}$, we have
\[
\begin{split}
h_{n-i}-h_i&=(-1)^i((-1)^n-1)b_1\binom{n}{i}\\
&=(-1)^i(\chi(\partial(M/T))-\chi(S^{n-1}))\binom{n}{i} \quad\text{for $0\le i\le n$}.
\end{split}
\]
These are generalized Dehn-Sommerville relations for
$\partial(M/T)$ (or for the simplicial poset $\P$), see \cite[p.
74]{stan96} or \cite[Theorem 7.44]{bu-pa02}.
\end{remark}

We will use the notations in Section~\ref{sectBettiNumbers} freely.
For a manifold $Q$ of dimension $n$ with corners (or faces), we
define the $f$-polynomial and $h$-polynomial of $Q$ by
\[
f_Q(t)=t^n+\sum_{i=0}^{n-1}f_i(Q)t^{n-1-i},\qquad h_Q(t)=f_Q(t-1)
\]
as usual.

\begin{lemma} \label{lemm:4-4}
The $h$-polynomials of $M'/T$, $M/T$, and $F$ have the relation
$h_{M'/T}(t) =h_{M/T}(t)+(t+1)h_F(t)-(t-1)^n$. Therefore
$$t^nh_{M'/T}(t^{-1})
=t^nh_{M/T}(t^{-1})+(1+t)t^{n-1}h_F(t^{-1})-(1-t)^n.$$
\end{lemma}

\begin{proof}
In the proof of Lemma~\ref{lemm:3-5} we observed that no facet of
$M'/T$ intersects with both $F_+$ and $F_-$.  This means that no
face of $M'/T$ intersects with both $F_+$ and $F_-$ because every
face of $M'/T$ is contained in some facet of $M'/T$.  Noting this
fact, one can find that
\[
f_i(M'/T)=f_i(M/T)+2f_{i-1}(F)+f_i(F)\quad\text{for $0\le i\le n-1$}
\]
where $F$ is the folded facet and $f_{n-1}(F)=0$.
Therefore,
\[
\begin{split}
f_{M'/T}(t)&=t^n+\sum_{i=0}^{n-1}f_i(M'/T)t^{n-1-i}\\
&=t^n+\sum_{i=0}^{n-1}f_i(M'/T)t^{i}+2\sum_{i=0}^{n-1}f_{i-1}(F)t^{n-1-i}+\sum_{i=0}^{n-2}f_i(F)t^{n-1-i}\\
&=f_{M/T}(t)+2f_F(t)+tf_F(t)-t^n.
\end{split}
\]
Replacing $t$ by $t-1$ in the identity above, we obtain the former
identity in the lemma. Replacing $t$ by $t^{-1}$ in the former
identity and multiplying the resulting identity by $t^n$, we
obtain the latter identity.
\end{proof}

\begin{proof}[Proof of Theorem~\ref{theo:4-1}]
Since $\sum_{i=0}^nh_i(M/T)t^i=t^nh_{M/T}(t^{-1})$,
Theorem~\ref{theo:4-1} is equivalent to
\begin{equation} \label{eq:4-8}
\sum_{i=0}^nb_{2i}(M)t^i=t^nh_{M/T}(t^{-1}) +b_1(M)(1+t^n-(1-t)^n).
\end{equation}
We shall prove \eqref{eq:4-8} by induction on $b_1(M)$.  The
identity \eqref{eq:4-8} is well-known when $b_1(M)=0$.  Suppose
that $k=b_1(M)$ is a positive integer and the identity
\eqref{eq:4-8} holds for $M'$ with $b_1(M')= k-1$.  Then
\[
\begin{split}
&\sum_{i=0}^nb_{2i}(M)t^i\\
=&1+t^n+\sum_{i=1}^{n-1} (b_{2i}(M')-b_{2i}(B)-b_{2i-2}(B))t^i\quad \text{(by Theorem~\ref{theo:3-1})}\\
=&\sum_{i=0}^nb_{2i}(M')t^i-(1+t)\sum_{i=0}^{n-1}b_{2i}(B)t^i+1+t^n\\
=&t^nh_{M'/T}(t^{-1})+b_1(M')(1+t^n-(1-t)^n)-(1+t)t^{n-1}h_{F}(t^{-1})+1+t^n\\
&\qquad\qquad\qquad\qquad\qquad\qquad\qquad\text{(by \eqref{eq:4-8} applied to $M'$)}\\
=&t^nh_{M/T}(t^{-1})+b_1(M)(1+t^n-(1-t)^n) \\
&\qquad\qquad\qquad\qquad\text{(by Lemma~\ref{lemm:4-4} and $b_1(M')=b_1(M)-1$),}
\end{split}
\]
proving \eqref{eq:4-8} for $M$. This completes the induction step
and the proof of Theorem~\ref{theo:4-1}.
\end{proof}

\section{Equivariant cohomology and face ring} \label{sectEquivar}

A torus manifold $M$ of dimension $2n$ is an orientable connected
closed smooth manifold with an effective smooth action of an
$n$-dimensional torus $T$ having a fixed point (\cite{ha-ma03}).
An orientable toric origami manifold with acyclic proper faces in
the orbit space has a fixed point, so it is a torus manifold. The
action of $T$ on $M$ is called \emph{locally standard} if every
point of $M$ has a $T$-invariant open neighborhood equivariantly
diffeomorphic to a $T$-invariant open set of a faithful
representation space of $T$.  Then the orbit space $M/T$ is a nice
manifold with corners\footnote{Faces of $M/T$ are defined using types of isotropy subgroups of the $T$-action on $M$.  The vertices in $M/T$ correspond to the $T$-fixed points in $M$ and \emph{nice} means that there are exactly $n$ facets (i.e., codimension-one faces) meeting at each vertex in $M/T$. \emph{A nice manifold with corners} is often called \emph{a manifold with faces}.}. The torus action on an orientable toric
origami manifold is locally standard. In this section, we study
the equivariant cohomology of a locally standard torus manifold
with acyclic proper faces of the orbit space.

We review some facts from \cite{ma-pa06}. Let $Q$ be a nice
manifold with corners of dimension $n$.  Let $\kr$ be a
ground commutative ring with unit. We denote by $G\vee H$ the unique minimal face of $Q$ that contains both
$G$ and $H$. The face ring $\kr[Q]$ of $Q$
is a graded ring defined by
\[
\kr[Q]:=\kr[v_F:\text{ $F$ a face}]/I_Q
\]
where $\deg v_F=2\codim F$ and $I_Q$ is the ideal generated by all
elements
\[
v_Gv_H-v_{G\vee H}\sum_{E\in G\cap H}v_E.
\]
The dual poset of the face poset of $Q$ is a
simplicial poset of dimension $n-1$ and its face ring over $\kr$
(see \cite[p.113]{stan96}) agrees with $\kr[Q]$. For each vertex
$p\in Q$, the restriction map $s_p$ is defined as the quotient map
\[
s_p\colon \kr[Q]\to \kr[Q]/(v_F:\ p\notin F)
\]
and it is proved in \cite[Proposition 5.5]{ma-pa06} that the image
$s_p(\kr[Q])$ is the polynomial ring
$\kr[v_{Q_{i_1}},\dots,v_{Q_{i_n}}]$ where $Q_{i_1},\dots,Q_{i_n}$
are the $n$ different facets containing $p$.

\newpage

\begin{lemma}[Lemma 5.6 in \cite{ma-pa06}] \label{lemm:1-1}
If every face of $Q$ has a vertex, then the sum $s=\oplus_{p}s_p$
of restriction maps over all vertices $p\in Q$ is a monomorphism
from $\kr[Q]$ to the sum of polynomial rings.
\end{lemma}

In particular, $\kr[Q]$ has no nonzero nilpotent element if every
face of $Q$ has a vertex.  It is not difficult to see that every
face of $Q$ has a vertex if every proper face of $Q$ is acyclic.

Let $M$ be a locally standard torus manifold. Then the orbit space
$M/T$ is a nice manifold with corners. Let $q\colon M\to M/T$ be
the quotient map.  Note that $M^\circ:=M-q^{-1}(\partial(M/T))$ is the
$T$-free part.  The projection $ET\times M\to M$ induces a map
$\bar q\colon ET\times_T M\to M/T$, where $ET$ denotes the total space of the universal principal $T$-bundle and $ET\times_T M$ denotes the orbit space of $ET\times M$ by the diagonal action of $T$ on $ET\times M$.  Similarly we have a map $\bar
q^\circ\colon ET\times_T M^\circ\to M^\circ/T$. The exact sequence
of the equivariant cohomology groups for the pair $(M,M^\circ)$ together
with the maps $\bar q$ and $\bar q^\circ$ produces the following
commutative diagram:
\begin{equation*}
\begin{CD}
H^*_T(M,M^\circ)@>\eta^*>> H^*_T(M)@>{\iota^*}>> H^*_T(M^\circ)\\
@. @A\bar q^*AA @AA{(\bar q^\circ)}^* A\\
@. H^*(M/T) @>\bar\iota^*>> H^*(M^\circ/T)
\end{CD}
\end{equation*}
where $\eta$, $\iota$ and $\bar\iota$ are the inclusions and $H^*_T(X,Y):=H^*(ET\times_T X,ET\times_TY)$ for a $T$-space $X$ and its $T$-subspace $Y$ as usual. Since
the action of $T$ on $M^\circ$ is free and $\bar\iota\colon M^\circ/T\to M/T$
is a homotopy equivalence, we have graded ring isomorphisms
\begin{equation} \label{eq:1-0-0-0}
H^*_T(M^\circ)\xrightarrow{((\bar q^\circ)^*)^{-1}} H^*(M^\circ/T)\xrightarrow{(\bar\iota^*)^{-1}} H^*(M/T)
\end{equation}
and the composition $\rho:=\bar q^*\circ (\bar\iota^*)^{-1}\circ
((\bar q^\circ)^*)^{-1}$, which is a graded ring homomorphism, gives
the right inverse of $\iota^*$, so the exact sequence above
splits. Therefore, $\eta^*$ and $\bar q^*$ are both injective and
\begin{equation} \label{eq:1-0-0}
H^*_T(M)= \eta^*(H^*_T(M,M^\circ))\oplus \rho(H^*_T(M^\circ)) \quad\text{as graded groups}.
\end{equation}
Note that both factors at the right hand side above are graded
subrings of $H^*_T(M)$ because $\eta^*$ and $\rho$ are both graded
ring homomorphisms.

Let $\P$ be the poset dual to the face poset of $M/T$ as before.
Then $\Z[\P]=\Z[M/T]$ by definition.

\begin{proposition} \label{prop:1-1}
Suppose every proper face of the orbit space $M/T$ is acyclic, and
the free part of the action gives a trivial principal bundle
$M^\circ\to M^\circ/T$. Then $H^*_T(M)\cong \Z[\P]\oplus \tilde H^*(M/T)$
as graded rings.
\end{proposition}

\begin{proof}
Let $R$ be the cone of $\partial(M/T)$ and let $M_R=M_R(\Lambda)$
be the $T$-space $R\times T/\sim$ where we use the characteristic
function $\Lambda$ obtained from $M$ for the identification
$\sim$.  Let $M^\circ_R$ be the $T$-free part of $M_R$.  Since the
free part of the action on $M$ is trivial, we have
$M-M^\circ=M_R-M^\circ_R$. Hence,
\begin{equation} \label{eq:1-0}
H^*_T(M,M^\circ)\cong H^*_T(M_R,M^\circ_R)\quad\text{as graded rings}
\end{equation}
by excision.  Since $H^*_T(M^\circ_R)\cong H^*(M^\circ_R/T)\cong H^*(R)$
and $R$ is a cone, $H^*_T(M^\circ_R)$ is isomorphic to the cohomology
of a point.  Therefore,
\begin{equation} \label{eq:1-1}
H^*_T(M_R,M^\circ_R)\cong H^*_T(M_R) \quad\text{as graded rings in positive degrees.}
\end{equation}

On the other hand, the dual decomposition on the geometric
realization $|\P|$ of $\P$ defines a face structure on the cone
$P$ of $\P$. Let $M_P=M_P(\Lambda)$ be the $T$-space $P\times
T/\sim$ defined as before. Then a similar argument to that in
\cite[Theorem 4.8]{da-ja91} shows that
\begin{equation} \label{eq:1-2}
H^*_T(M_P)\cong \Z[\P] \quad\text{as graded rings}
\end{equation}
(this is mentioned as Proposition 5.13 in \cite{ma-pa06}).  Since
every face of $P$ is a cone, one can construct a face preserving
degree one map from $R$ to $P$ which induces an equivariant map
$f\colon M_R\to M_P$.  Then a similar argument to the proof of
Theorem 8.3 in \cite{ma-pa06} shows that $f$ induces a graded ring
isomorphism
\begin{equation} \label{eq:1-3}
f^*\colon H^*_T(M_P)\xrightarrow{\cong} H^*_T(M_R)
\end{equation}
since every proper face of $R$ is acyclic. It follows from
\eqref{eq:1-0}, \eqref{eq:1-1}, \eqref{eq:1-2} and \eqref{eq:1-3}
that
\begin{equation} \label{eq:1-3-1}
H^*_T(M,M^\circ) \cong \Z[\P] \quad\text{as graded rings in positive
degrees.}
\end{equation}

Thus, by \eqref{eq:1-0-0-0} and \eqref{eq:1-0-0} it suffices to
prove that the cup product of every $a\in \eta^*(H^*_T(M,M^\circ))$ and
every $b\in \rho(\tilde H^*_T(M^\circ))$ is trivial.  Since
$\iota^*(a)=0$ (as $\iota^*\circ \eta^*=0$), we have
$\iota^*(a\cup b)=\iota^\ast(a)\cup \iota^\ast(b)=0$ and hence $a\cup b$
lies in $\eta^*(H^*_T(M,M^\circ))$.  Since $\rho(H^*_T(M^\circ))\cong
H^*(M/T)$ as graded rings by \eqref{eq:1-0-0-0} and $H^m(M/T)=0$
for a sufficiently large $m$, $(a\cup b)^m=\pm a^m\cup b^m=0$.
However, we know that $a\cup b\in \eta^*(H^*_T(M,M^\circ))$ and
$\eta^*(H^*_T(M,M^\circ))\cong \Z[\P]$ in positive degrees by
\eqref{eq:1-3-1}.  Since $\Z[\P]$ has no nonzero nilpotent element
as remarked before, $(a\cup b)^m=0$ implies $a\cup b=0$.
\end{proof}

As discussed in \cite[Section 6]{ma-pa06}, there is a homomorphism
\begin{equation} \label{eq:1-4}
\varphi\colon \Z[\P]=\Z[M/T] \to H^*_T(M)/\text{$H^*(BT)$-torsions}.
\end{equation}
In fact, $\varphi$ is defined as follows.  For a codimension $k$
face $F$ of $M/T$,  $q^{-1}(F)=:M_F$ is a connected closed
$T$-invariant submanifold of $M$ of codimension $2k$, and
$\varphi$ assigns $v_F\in \Z[M/T]$ to the equivariant Poincar\'e
dual $\tau_F \in H^{2k}_T(M)$ of $M_F$. One can see that $\varphi$
followed by the restriction map to $H^*_T(M^T)$ can be identified
with the map $s$ in Lemma~\ref{lemm:1-1}.  Therefore, $\varphi$ is
injective if every face of $Q$ has a vertex as mentioned in
\cite[Lemma 6.4]{ma-pa06}.

\begin{proposition} \label{prop:1-2}
Let $M$ be a torus manifold with a locally standard torus action. If every proper face of
$M/T$ is acyclic and the free part of the action gives a trivial
principal bundle, then the $H^*(BT)$-torsion submodule of
$H^*_T(M)$ agrees with $\bar q^*(\tilde H^*(M/T))$, where $\bar
q\colon ET\times_T M\to M/T$ is the map mentioned before.
\end{proposition}

\begin{proof} First we prove that all elements in $\bar q^*(\tilde H^*(M/T))$
are $H^*(BT)$-torsions. We consider the following commutative diagram:
\[
\begin{CD}
 H^*_T(M)@>\psi^*>>  H^*_T(M^T)\\
@A \bar q^* AA  @AA  A\\
 H^*(M/T) @>\bar\psi^*>>  H^*(M^T)
\end{CD}
\]
where the horizontal maps $\psi^*$ and $\bar \psi^*$ are
restrictions to $M^T$ and the right vertical map is the
restriction of $\bar q^*$ to $M^T$.  Since $M^T$ is isolated,
$\bar\psi^*(\tilde H^*(M/T))$ vanishes.  This together with the
commutativity of the above diagram shows that $\bar q^*(\tilde
H^*(M/T))$ maps to zero by $\psi^*$.  This means that  all elements in $\bar
q^*(\tilde H^*(M/T))$ are $H^*(BT)$-torsions because the kernel of
$\psi^*$ are $H^*(BT)$-torsions by the Localization Theorem in
equivariant cohomology.

On the other hand, since every face of $M/T$ has a vertex, the map
$\varphi$ in \eqref{eq:1-4} is injective as remarked above. Hence, by Proposition~\ref{prop:1-1}, there are no other
$H^*(BT)$-torsion elements.
\end{proof}

We conclude this section with observation on the orientability of
$M/T$.

\begin{lemma} \label{lemm:1-3}
Let $M$ be a closed smooth manifold of dimension $2n$ with a
locally standard smooth action of the $n$-dimensional torus $T$.
Then $M/T$ is orientable if and only if $M$ is.
\end{lemma}

\begin{proof}
Since $M/T$ is a manifold with corners and $M^\circ/T$ is its
interior, $M/T$ is orientable if and only if $M^\circ/T$ is. On
the other hand, $M$ is orientable if and only if $M^\circ$ is.
Indeed, since the complement of $M^\circ$ in $M$ is the union of
finitely many codimension-two submanifolds, the inclusion
$\iota\colon M^\circ\to M$ induces an epimorphism on their
fundamental groups and hence on their first homology groups with
$\Z/2$-coefficients.  Then it induces a monomorphism
$\iota^*\colon H^1(M;\Z/2)\to H^1(M^\circ;\Z/2)$ since
$H^1(X;\Z/2)=\Hom(H_1(X;\Z/2);\Z/2)$.  Since
$\iota^*(w_1(M))=w_1(M^\circ)$ and $\iota^*$ is injective,
$w_1(M)=0$ if and only if $w_1(M^\circ)=0$. This means that $M$ is
orientable if and only if $M^\circ$ is.

Thus, it suffices to prove that $M^\circ/T$ is orientable if and
only if $M^\circ$ is.  But, since $M^\circ/T$ can be regarded
as the quotient of an iterated free action of $S^1$, it suffices to
prove the following general fact: for a principal $S^1$-bundle
$\pi\colon E\to B$ where $E$ and $B$ are both smooth manifolds,
$B$ is orientable if and only if $E$ is. First we note that the
tangent bundle of $E$ is isomorphic to the Whitney sum of the
tangent bundle along the fiber $\tau_fE$ and the pullback of the
tangent bundle of $B$ by $\pi$.  Since the free action of $S^1$ on
$E$ yields a nowhere zero vector field along the fibers, the line
bundle $\tau_fE$ is trivial.  Therefore
\begin{equation} \label{eq:1-5}
w_1(E)=\pi^*(w_1(B)).
\end{equation}
We consider the Gysin exact sequence for our $S^1$-bundle:
\[
\cdots\to H^{-1}(B;\Z/2)\to H^{1}(B;\Z/2)\xrightarrow{\pi^*} H^1(E,\Z/2)\to H^{0}(B;\Z/2)\to\cdots.
\]
Since $H^{-1}(B;\Z/2)=0$, the exact sequence above tells us that the map
$\pi^*\colon H^1(B;\Z/2) \to H^{1}(E;\Z/2)$ is injective.  This
together with \eqref{eq:1-5} shows that $w_1(E)=0$ if and only if
$w_1(B)=0$, proving the desired fact.
\end{proof}

\section{Serre spectral sequence} \label{sectSerreSpec}

Let $M$ be an orientable toric origami manifold $M$ of dimension
$2n$ such that every proper face of $M/T$ is acyclic. Note that
$M^\circ/T$ is homotopy equivalent to a graph, hence does not admit
nontrivial torus bundles. Thus the free part of the action gives a
trivial principal bundle $M^\circ\to M^\circ/T$, and we may apply the
results of the previous section.

We consider the Serre spectral sequence of the fibration
$\pi\colon ET\times_T M\to BT$. Since $BT$ is simply connected and
both $H^*(BT)$ and $H^*(M)$ are torsion free by
Theorem~\ref{theo:3-1}, the $E_2$-terms are given as follows:
$$E_2^{p,q}=H^p(BT;H^q(M))=H^p(BT)\otimes H^q(M).$$
Since $H^{odd}(BT)=0$ and $H^{2i+1}(M)=0$ for $1\le i\le n-2$ by
Theorem~\ref{theo:3-1},
\begin{equation} \label{eq:2-1-1}
\text{$E_2^{p,q}$ with $p+q$ odd vanishes unless $p$ is even and
$q=1$ or $2n-1$.}
\end{equation}

We have differentials
\[
\to E_r^{p-r,q+r-1}\xrightarrow{d_r^{p-r,q+r-1}}E_r^{p,q}\xrightarrow{d_r^{p,q}}E_r^{p+r,q-r+1}\to
\]
and
$$E_{r+1}^{p,q}=\ker d_r^{p,q}/\im d_r^{p-r,q+r-1}.$$
We will often abbreviate $d_r^{p,q}$ as $d_r$ when $p$ and $q$ are
clear in the context. Since
\[
d_r(u\cup v)=d_ru\cup v+(-1)^{p+q}u\cup d_rv \quad\text{for $u\in E_r^{p,q}$ and $v\in E_r^{p',q'}$}
\]
and $d_r$ is trivial on $E_r^{p,0}$ and $E_r^{p,0}=0$ for odd $p$,
\begin{equation} \label{eq:2-1-10}
\text{$d_r$ is an $H^*(BT)$-module map.}
\end{equation}
Note that $E_r^{a,b}=0$ if either $a<0$ or $b<0$. Accordingly,
\begin{equation} \label{eq:2-1-2}
\text{$E^{p,q}_r=E^{p,q}_\infty$\quad if $p<r$ and $q+1<r$.}
\end{equation}

There is a filtration of subgroups
\[
H_T^{m}(M)=\F^{0,m}\supset \F^{1,m-1}\supset \dots\supset
\F^{m-1,1}\supset \F^{m,0}\supset\F^{m+1,-1}=\{0\}
\]
such that
\begin{equation} \label{eq:2-1-3}
\F^{p,m-p}/\F^{p+1,m-p-1}= E_\infty^{p,m-p}\quad \text{for
$p=0,1,\dots,m$}.
\end{equation}

There are two edge homomorphisms.
One edge homomorphism
\[
H^p(BT)=E_2^{p,0}\to E_3^{p,0}\to \dots \to E_\infty^{p,0}\subset H^p_T(M)
\]
agrees with $\pi^*\colon H^*(BT)\to H^*_T(M)$. Since
$M^T\not=\emptyset$, one can construct a cross section of the
fibration  $\pi\colon ET\times_T M\to BT$ using a fixed point in
$M^T$. So $\pi^*$ is injective and hence
\begin{equation} \label{eq:2-0}
\text{$d_r\colon E_r^{p-r,r-1}\to E_r^{p,0}$ is trivial for every $r\ge 2$ and $p\ge 0$,}
\end{equation}
which is equivalent to $E_2^{p,0}=E_\infty^{p,0}$.
The other edge homomorphism
\[
H_T^q(M)\twoheadrightarrow E_\infty^{0,q}\subset \dots \subset E_3^{0,q}\subset E_2^{0,q}=H^q(M)
\]
agrees with the restriction homomorphism $\iota^*\colon
H^q_T(M)\to H^q(M)$. Therefore, $\iota^*$ is surjective if and
only if the differential $d_r\colon E_r^{0,q}\to E_r^{r,q-r+1}$ is
trivial for every $r\ge 2$.

We shall investigate the restriction homomorphism $\iota^*\colon
H^q_T(M)\to H^q(M)$.  Since $M/T$ is homotopy equivalent to the
wedge of $b_1(M)$ circles, $H^q_T(M)$ vanishes unless $q$ is $1$
or even by Proposition~\ref{prop:1-1} while $H^q(M)$ vanishes
unless $q$ is $1, 2n-1$ or even in between $0$ and $2n$ by
Theorem~\ref{theo:3-1}.

\begin{lemma} \label{lemm:2-1}
The homomorphism $\iota^*\colon H^1_T(M)\to H^1(M)$ is an isomorphism (so
$H^1(M)\cong H^1(M/T)$ by Proposition~\ref{prop:1-1}).
\end{lemma}

\begin{proof}
By \eqref{eq:2-0}, $$d_2\colon E_2^{0,1} = H^1(M)\to
E_2^{2,0}=H^2(BT)$$ is trivial.  Therefore
$E_2^{0,1}=E_\infty^{0,1}$.  On the other hand,
$E_\infty^{1,0}=E_2^{1,0}=H^1(BT)=0$.  These imply the lemma.
\end{proof}

Since $H^{2n-1}_T(M)=0$, the homomorphism $\iota^*\colon H^{2n-1}_T(M)\to
H^{2n-1}(M)$ cannot be surjective unless $H^{2n-1}(M)=0$.

\begin{lemma} \label{lemm:2-1-1}
The homomorphism $\iota^*\colon H^{2j}_T(M)\to H^{2j}(M)$ is surjective except for
$j=1$ and the rank of the cokernel of $\iota^*$ for $j=1$ is
$nb_1(M)$.
\end{lemma}

\begin{proof}
Since $\dim M=2n$, we may assume $1\le j\le n$.

First we treat the case where $j=1$. Since $H^3_T(M)=0$,
$E_\infty^{2,1}=0$ by \eqref{eq:2-1-3} and
$E_\infty^{2,1}=E_3^{2,1}$ by \eqref{eq:2-1-2}.  This together
with \eqref{eq:2-0} implies that
\begin{equation} \label{eq:2-1-9}
d_2\colon H^2(M)=E_2^{0,2}\to E^{2,1}_2=H^2(BT)\otimes H^1(M) \quad\text{is surjective}.
\end{equation}
Moreover $d_3\colon E_3^{0,2}=\ker d_2\to E_3^{3,0}$ is trivial
since $E_3^{3,0}=0$.  Therefore, $E_3^{0,2}=E_\infty^{0,2}$ by
\eqref{eq:2-1-2}. Since $E_\infty^{0,2}$ is the image of
$\iota^*\colon H_T^2(M)\to H^2(M)$, the rank of
$H^2(M)/\iota^*(H^2_T(M))$ is $nb_1(M)$ by \eqref{eq:2-1-9}.

Suppose that $2\le j\le n-1$. We need to prove that the
differentials
\[
d_r\colon E_r^{0,2j}\to E_r^{r,2j-r+1}
\]
are all trivial.  In fact, the target group $E_r^{r,2j-r+1}$
vanishes.  This follows from \eqref{eq:2-1-1} unless $r=2j$.  As
for the case $r=2j$, we note that
\begin{equation} \label{eq:2-1-8}
d_2\colon E_2^{p,2}\to E_2^{p+2,1} \quad \text{is surjective for $p\ge 0$,}
\end{equation}
which follows from \eqref{eq:2-1-10} and \eqref{eq:2-1-9}.
Therefore $E_3^{p+2,1}=0$ for $p\ge 0$, in particular
$E_{r}^{r,2j-r+1}=0$ for $r=2j$ because $j\ge 2$.  Therefore
$\iota^*\colon H^{2j}_T(M)\to H^{2j}(M)$ is surjective for $2\le
j\le n-1$.

The remaining case $j=n$ can be proved directly, namely without
using the Serre spectral sequence. Let $x$ be a $T$-fixed point of
$M$ and let $\varphi\colon x\to M$ be the inclusion map.  Since
$M$ is orientable and $\varphi$ is $T$-equivariant, the
equivariant Gysin homomorphism $\varphi_!\colon H_T^0(x)\to
H_T^{2n}(M)$ can be defined and $\varphi_!(1)\in H_T^{2n}(M)$
restricts to the ordinary Gysin image of $1\in H^0(x)$, that is
the cofundamental class of $M$.  This implies the surjectivity of
$\iota^*\colon H^{2n}_T(M)\to H^{2n}(M)$ because $H^{2n}(M)$ is an
infinite cyclic group generated by the cofundamental class.
\end{proof}

\section{Towards the ring structure}\label{sectTowardsRing}

Let $\pi\colon ET\times_TM\to BT$ be the projection.  Note that
$\pi^*(H^{2}(BT))$ maps to zero by the restriction homomorphism
$\iota^*\colon H^*_T(M)\to H^*(M)$. Hence, $\iota^*$ induces a graded
ring homomorphism
\begin{equation} \label{eq:5-0}
\bar\iota^*\colon H^*_T(M)/(\pi^*(H^{2}(BT)))\to H^*(M)
\end{equation}
which is surjective except in degrees $2$ and $2n-1$ by
Lemma~\ref{lemm:2-1-1} (and bijective in degree $1$ by
Lemma~\ref{lemm:2-1}).  Here $(\pi^*(H^2(BT)))$ denotes the ideal
in $H^*_T(M)$ generated by $\pi^*(H^2(BT))$. The purpose of this
section is to prove the following.

\begin{proposition} \label{prop:5-1}
The map $\bar\iota^*$ in \eqref{eq:5-0} is an isomorphism except in
degrees $2$, $4$ and $2n-1$. Moreover, the rank of the cokernel of
$\bar\iota^*$ in degree $2$ is $nb_1(M)$ and the rank of the
kernel of $\bar\iota^*$ in degree $4$ is $\binom{n}{2}b_1(M)$.
\end{proposition}

The rest of this section is devoted to the proof of
Proposition~\ref{prop:5-1}.  We recall the following result, which
was proved by Schenzel (\cite{sche81}, \cite[p.73]{stan96}) for
Buchsbaum simplicial complexes  and generalized to Buchsbaum
simplicial posets by Novik-Swartz (\cite[Proposition
6.3]{no-sw09}).
There are several equivalent definitions for Buchsbaum simplicial complexes (see \cite[p.73]{stan96}).  A convenient one for us would be that a finite simplicial complex $\Delta$ is Buchsbaum (over a field $\Bbbk$) if $H_i(|\Delta|,|\Delta|\backslash\{p\};\Bbbk)=0$ for all $p\in |\Delta|$ and all $i<\dim |\Delta|$, where $|\Delta|$ denotes the realization of $\Delta$.  In particular, a triangulation $\Delta$ of a manifold is Buchsbaum over any field $\Bbbk$.   A simplicial poset is a (finite) poset $P$ that has a unique minimal element, $\hat 0$, and
such that for every $\tau\in P$, the interval $[\hat 0,\tau]$ is a Boolean algebra.  The face poset of a simplicial complex is a simplicial poset and one has the realization $|P|$ of $P$ where $|P|$ is a regular CW complex, all of whose closed cells are simplices corresponding to the intervals $[\hat 0,\tau]$.  A simplicial poset $P$ is Buchsbaum (over $\Bbbk$) if its order complex $\Delta(\overline P)$ of the poset $\overline P=P\backslash\{\hat 0\}$ is Buchsbaum (over $\Bbbk$).  Note that $|\Delta(\overline P)|=|P|$ as spaces since $|\Delta(\overline P)|$ is the barycentric subdivision of $|P|$. See \cite{no-sw09} and \cite{stan96} for more details.

\begin{theorem}[Schenzel, Novik-Swartz] \label{theo:5-1}
Let $\Delta$ be a Buchsbaum simplicial poset of dimension $n-1$ over a field $\Bbbk$, $\Bbbk[\Delta]$ be the face ring of $\Delta$
and let $\theta_1,\dots,\theta_n\in \Bbbk[\Delta]_1$ be a linear
system of parameters. Then
\[
\begin{split}
F(\Bbbk[\Delta]/(\theta_1,\dots,\theta_n),t)=&(1-t)^nF(\Bbbk[\Delta],t)\\
&+\sum_{j=1}^n\binom{n}{j}\Big(\sum_{i=-1}^{j-2}(-1)^{j-i}\dim_\Bbbk
\tilde H_{i}(\Delta)\Big)t^j
\end{split}
\]
where $F(M,t)$ denotes the Hilbert series of a graded module $M$.
\end{theorem}

As is well-known, the Hilbert series of the face ring $\Bbbk[\Delta]$ satisfies
\[
(1-t)^nF(\Bbbk[\Delta],t)=\sum_{i=0}^nh_it^i.
\]
We define $h_i'$ for $i=0,1,\dots,n$
by
\[
F(\Bbbk[\Delta]/(\theta_1,\dots,\theta_n),t)=\sum_{i=0}^nh_i't^i,
\]
following \cite{no-sw09}.

\begin{remark}
Novik-Swartz \cite{no-sw09} introduced
\[
h_i'':=h_i'-\binom{n}{j}\dim_\Bbbk \tilde
H_{j-1}(\Delta)=h_j+\binom{n}{j}\Big(\sum_{i=-1}^{j-1}(-1)^{j-i}\dim_\Bbbk
\tilde H_{i}(\Delta)\Big)
\]
for $1\le i\le n-1$ and showed that $h''_j\ge 0$ and
$h''_{n-j}=h''_j$ for $1\le j\le n-1$.
\end{remark}

We apply Theorem~\ref{theo:5-1} to our simplicial poset $\P$ which
is dual to the face poset of $\partial(M/T)$.  For that we need to
know the homology of the geometric realization $|\P|$ of $\P$.
First we show that $|\P|$ has the same homological features as
$\partial(M/T)$.

\begin{lemma} \label{lemm:samehomology}
The simplicial poset $\P$ is Buchsbaum, and $|\P|$ has the same
homology as $\partial(M/T)$.
\end{lemma}

\begin{proof}
We give a sketch of the proof. Details can be found in \cite[Lemma
3.14]{ayze14}. There is a dual face structure on $|\P|$, and there
exists a face preserving map $g\colon \partial(M/T)\to |\P|$
mentioned in the proof of Proposition \ref{prop:1-1}. Let $F$ be a
proper face of $M/T$ and $F'$ the corresponding face of $|\P|$. By
induction on $\dim F$ we can show that $g$ induces the
isomorphisms $g_*\colon H_*(\partial F)\xrightarrow{\cong}
H_*(\partial F')$, $g_*\colon H_*(F)\xrightarrow{\cong} H_*(F')$,
and $g_*\colon H_*(F,\partial F)\xrightarrow{\cong}
H_*(F',\partial F')$. Since $F$ is an acyclic orientable manifold
with boundary, we deduce by Poincar\'{e}-Lefschetz duality that
$H_*(F',\partial F')\cong H_*(F,\partial F)$ vanishes except in
degree $\dim F$. Note that $F'$ is a cone over $\partial F'$ and
$\partial F'$ is homeomorphic to the link of a nonempty simplex of
$\P$. Thus the links of nonempty simplices of $\P$ are homology
spheres, and $\P$ is Buchsbaum \cite[Prop.6.2]{no-sw09}. Finally,
$g$ induces an isomorphism of spectral sequences corresponding to
skeletal filtrations of $\partial(M/T)$ and $|\P|$, thus induces
an isomorphism $g_*\colon H_*(\partial
(M/T))\xrightarrow{\cong}H_*(|\P|)$.
\end{proof}


\begin{lemma} \label{lemm:5-1}
$|\P|$ has the same homology as $S^{n-1}\sharp b_1(S^1\times
S^{n-2})$ (the connected sum of $S^{n-1}$ and $b_1$ copies of
$S^1\times S^{n-2}$).
\end{lemma}

\begin{proof}
By Lemma~\ref{lemm:samehomology} we only need to prove that $\partial(M/T)$
has the same homology groups as $S^{n-1}\sharp b_1(S^1\times
S^{n-2})$.  Since $M/T$ is homotopy equivalent to a wedge of
circles, $H^i(M/T)=0$ for $i\ge 2$ and hence the homology exact
sequence of the pair $(M/T,\partial(M/T))$ shows that
\[
H_{i+1}(M/T,\partial(M/T))\cong H_{i}(\partial(M/T)) \quad\text{for $i\ge 2$}.
\]
On the other hand, $M/T$ is orientable by Lemma~\ref{lemm:1-3} and
hence
\[
H_{i+1}(M/T,\partial(M/T))\cong H^{n-i-1}(M/T)
\]
by Poincar\'e--Lefschetz duality, and $H^{n-i-1}(M/T)=0$ for
$n-i-1\ge 2$.  These show that
\begin{equation*} \label{eq:5-0-0}
H_i(\partial(M/T))=0\quad\text{for $2\le i\le n-3$}.
\end{equation*}
Thus it remains to study $H_i(\partial(M/T))$ for $i=0,1,n-2,n-1$
but since $\partial(M/T)$ is orientable (because $M/T$ is orientable), it
suffices to show
\begin{equation} \label{eq:5-0-1}
H_i(\partial(M/T))\cong H_i(S^{n-1}\sharp b_1(S^1\times
S^{n-2}))\quad\text{for $i=0,1$}.
\end{equation}

When $n\ge 3$, $S^{n-1} \sharp b_1(S^1\times S^{n-2})$ is
connected, so \eqref{eq:5-0-1} holds for $i=0$ and $n\ge 3$.
Suppose that $n\ge 4$.  Then $H^{n-2}(M/T)=H^{n-1}(M/T)=0$, so the
cohomology exact sequence for the pair $(M/T,\partial(M/T))$ shows
that $H^{n-2}(\partial(M/T))\cong H^{n-1}(M/T,\partial(M/T))$ and
hence $H_1(\partial(M/T))\cong H_1(M/T)$ by Poincar\'e--Lefschetz
duality.  Since $M/T$ is homotopy equivalent to a wedge of $b_1$
circles, this proves \eqref{eq:5-0-1} for $i=1$ and $n\ge 4$. Assume that
$n=3$. Then $H_1(M/T,\partial(M/T))\cong H^2(M/T)=0$.  We also know that
$H_2(M/T)=0$. The homology exact sequence for the pair
$(M/T,\partial(M/T))$ yields a short exact sequence
\[
0\to H_2(M/T,\partial(M/T))\to H_1(\partial(M/T))\to H_1(M/T)\to 0.
\]
Here $H_2(M/T,\partial(M/T)) \cong H^1(M/T)$ by
Poincar\'e--Lefschetz duality.  Since $M/T$ is homotopy equivalent
to a wedge of $b_1$ circles, this implies \eqref{eq:5-0-1} for
$i=1$ and $n=3$.

It remains to prove \eqref{eq:5-0-1} when $n=2$.  We use induction
on $b_1$.  The assertion is true when $b_1=0$.  Suppose that
$b_1=b_1(M/T)\ge 1$.  We cut $M/T$ along a fold so that
$b_1(M'/T)=b_1(M/T)-1$, where $M'$ is the toric origami manifold
obtained from the cut, see Section \ref{sectBettiNumbers}. Then
$b_0(\partial(M'/T))=b_0(\partial(M/T))-1$. Since \eqref{eq:5-0-1}
holds for $\partial(M'/T)$ by induction assumption, this
observation shows that \eqref{eq:5-0-1} holds for $\partial(M/T)$.
\end{proof}

\begin{lemma} \label{lemm:5-2}
For $n\ge 2$, we have
\[
\sum_{i=0}^nh_i't^i= \sum_{i=0}^nb_{2i}t^i-nb_1t+\binom{n}{2}b_1t^2.
\]
\end{lemma}

\begin{proof}
By Lemma~\ref{lemm:5-1}, for $n\ge 4$, we have
\[
\dim\tilde H_i(\P)=\begin{cases} b_1 \quad&\text{if $i=1,\ n-2$},\\
1 \quad&\text{if $i=n-1$},\\
0\quad &\text{otherwise}.
\end{cases}
\]
Hence
\[
\sum_{i=-1}^{j-2}(-1)^{j-i}\dim \tilde H_{i}(\P)=
\begin{cases} 0\quad&\text{if $j=1,2$},\\
(-1)^{j-1}b_1 \quad&\text{if $3\le j\le n-1$},\\
((-1)^{n-1}+1)b_1\quad&\text{if $j=n$}.
\end{cases}
\]
Then, it follows from Theorem~\ref{theo:5-1} that
\begin{equation*} \label{eq:5-3}
\begin{split}
\sum_{i=0}^nh_i't^i=& \sum_{i=0}^nh_it^i+\sum_{j=3}^{n-1}(-1)^{j-1}b_1\binom{n}{j}t^j+((-1)^{n-1}+1)b_1t^n\\
=&  \sum_{i=0}^nh_it^i-b_1(1-t)^n+b_1(1-nt+\binom{n}{2}t^2)+b_1t^n\\
=&\sum_{i=0}^nh_it^i+b_1(1+t^n-(1-t)^n)-nb_1t+\binom{n}{2}b_1t^2\\
=&\sum_{i=0}^nb_{2i}t^i-nb_1t+\binom{n}{2}b_1t^2
\end{split}
\end{equation*}
where we used Theorem~\ref{theo:4-1} at the last identity.  This
proves the lemma when $n\ge 4$.  When $n=3$,
\[
\dim\tilde H_i(\P)=\begin{cases} 2b_1 \quad&\text{if $i=1$},\\
1 \quad&\text{if $i=2$},\\
0\quad &\text{otherwise},
\end{cases}
\]
and the same argument as above shows that the lemma still holds
for $n=3$.  When $n=2$,
\[
\dim\tilde H_i(\P)=\begin{cases} b_1 \quad&\text{if $i=0$},\\
b_1+1 \quad&\text{if $i=1$},\\
0\quad &\text{otherwise},
\end{cases}
\]
and the same holds in this case too.
\end{proof}

\begin{remark}
One can check that
\[
\sum_{i=1}^{n-1}h_i''t^i=\sum_{i=1}^{n-1}b_{2i}t^i-nb_1(t+t^{n-1}).
\]
Therefore, $h_i''=h_i''(\P)$ is not necessarily equal to
$b_{2i}=b_{2i}(M)$ although both are symmetric.  This is not surprising
because $h_i''$ depends only on the boundary of $M/T$.  It would
be interesting to ask whether $h_i''(\P)\le b_{2i}(M)$ 
when the face poset of $\partial(M/T)$ is dual to $\P$ and
whether the equality can be attained for some such $M$ ($M$ may
depend on $i$).
\end{remark}

Now we are ready to prove Proposition \ref{prop:5-1}.
\begin{proof}[Proof of Proposition~\ref{prop:5-1}]
At first we suppose that
$\Bbbk$ is a field. By Proposition~\ref{prop:1-1}, we have
$\Z[\P]=H^{even}_T(M)$. The images of ring generators of
$H^*(BT;\Bbbk)$ by $\pi^*$ provide an h.s.o.p.
$\theta_1,\dots,\theta_n$ in $H^{even}_T(M;\Bbbk)=\Bbbk[\P]$. This fact
simply follows from the characterization of homogeneous systems of
parameters in face rings given by \cite[Th.5.4]{bu-pa04}. Thus we
have
\begin{equation} \label{eq:5-4}
F(H^{even}_T(M;\Bbbk)/(\pi^*(H^{2}(BT;\Bbbk))),t)=\sum_{i=0}^nb_{2i}(M)t^i-nb_1t+\binom{n}{2}b_1t^2
\end{equation}
by Lemma~\ref{lemm:5-2}.  Moreover, the graded ring homomorphism
in \eqref{eq:5-0}
\begin{equation} \label{eq:5-4-1}
\bar\iota^*\colon
\Bbbk[\P]/(\theta_1,\dots,\theta_n)=H^{even}_T(M;\Bbbk)/(\pi^*(H^{2}(BT;\Bbbk)))\to
H^{even}(M;\Bbbk)
\end{equation}
is surjective except in degree $2$ as remarked at the beginning of
this section. Therefore, the identity \eqref{eq:5-4} implies that
$\bar\iota^*$ in \eqref{eq:5-4-1} is an isomorphism except in
degrees $2$ and $4$.  Finally, the rank of the cokernel of
$\bar\iota^*$ in degree $2$ is $nb_1(M)$ by Lemma~\ref{lemm:2-1-1}
and the rank of the kernel of $\bar\iota^*$ in degree $4$ is
$\binom{n}{2}b_1$ by \eqref{eq:5-4}, proving
Proposition~\ref{prop:5-1} over fields.

Now we explain the case
$\Bbbk=\Z$.
The map $\pi^*\colon H^*(BT;\Bbbk)\to H_T^*(M;\Bbbk)$ coincides with the
map $\pi^*\colon H^*(BT;\Z)\to H_T^*(M;\Z)$ tensored with $\Bbbk$,
since both $H^*(BT;\Z)$ and $H_T^*(M;\Z)$ are $\Z$-torsion free.
In particular, the ideals $(\pi^*(H^2(BT;\Bbbk)))$ and
$(\pi^*(H^2(BT;\Z))\otimes \Bbbk)=(\pi^*(H^2(BT;\Z)))\otimes \Bbbk$
coincide in $H_T^*(M;\Bbbk)\cong H_T^*(M;\Z)\otimes \Bbbk$. Consider the
exact sequence
\[
(\pi^*(H^2(BT;\Z)))\to H_T^*(M;\Z)\to
H_T^*(M;\Z)/(\pi^*(H^2(BT;\Z)))\to 0
\]
The functor $-\otimes \Bbbk$ is right exact, thus the sequence
\[\begin{split}
&(\pi^*(H^2(BT;\Z)))\otimes \Bbbk\to H_T^*(M;\Z)\otimes \Bbbk \\
&\qquad\qquad\qquad\qquad\qquad\qquad\to
H_T^*(M;\Z)/(\pi^*(H^2(BT;\Z)))\otimes \Bbbk\to 0
\end{split}\]
is exact. These considerations show that
\[
H_T^*(M;\Z)/(\pi^*(H^2(BT;\Z)))\otimes \Bbbk \cong
H_T^*(M;\Bbbk)/(\pi^*(H^2(BT;\Bbbk)))
\]
Finally, the map
\[
\bar\iota^*\colon H_T^*(M;\Bbbk)/(\pi^*(H^2(BT;\Bbbk)))\to H^*(M,\Bbbk)
\]
coincides (up to isomorphism) with the map
\[
\bar\iota^*\colon H_T^*(M;\Z)/(\pi^*(H^2(BT;\Z)))\to H^*(M,\Z),
\]
tensored with $\Bbbk$. The statement of Proposition \ref{prop:5-1}
holds for any field thus holds for~$\Z$.
\end{proof}

We conclude this section with some observations on the kernel of
$\bar\iota^*$ in degree $4$ from the viewpoint of the Serre
spectral sequence.  Recall
\[
H^4_T(M)=\F^{0,4}\supset \F^{1,3}\supset \F^{2,2}\supset
\F^{3,1}\supset \F^{4,0}\supset \F^{5,-1}=0
\]
where $\F^{p,q}/\F^{p+1,q-1}=E^{p,q}_\infty$. Since
$E^{p,q}_2=H^p(BT)\otimes H^q(M)$, we have $E_\infty^{p,q}=0$ for $p$ odd.
Therefore,
\[
\rank H^4_T(M)=\rank E_\infty^{0,4}+\rank E_\infty^{2,2}+\rank
E_\infty^{4,0},
\]
where we know $E_\infty^{0,4}=E_2^{0,4}= H^4(M)$ and
$E_\infty^{4,0}=E_2^{4,0}=H^4(BT)$.  As for $E_\infty^{2,2}$, we
recall that
\begin{equation*}
d_2\colon E_2^{p,2}\to E^{p+2,1}_2 \quad\text{is surjective for every $p\ge 0$}
\end{equation*}
by \eqref{eq:2-1-8}.  Therefore, noting $H^3(M)=0$, one sees
$E_3^{2,2}=E_\infty^{2,2}$.  It follows that
\[
\rank E_\infty^{2,2}=\rank E_2^{2,2}-\rank
E_2^{4,1}=nb_2-\binom{n+1}{2}b_1.
\]
On the other hand, $\rank E_\infty^{0,2}=b_2-nb_1$ and there is a
product map
\[
\varphi\colon E_\infty^{0,2}\otimes E_\infty^{2,0}\to
E_\infty^{2,2}.
\]
The image of this map lies in the ideal $(\pi^*(H^2(BT))$ and the
rank of the cokernel of this map is
\[
nb_2-\binom{n+1}{2}b_1-n(b_2-nb_1)=\binom{n}{2}b_1.
\]
Therefore
\[
\rank E_\infty^{0,4}+\rank \coker\varphi=b_4+\binom{n}{2}b_1
\]
which agrees with the coefficient of $t^2$ in
$F(H^{even}_T(M)/(\pi^*(H^{2}(BT))),t)$ by \eqref{eq:5-4}. This
suggests that the cokernel of $\varphi$ could correspond to the
kernel of $\bar\iota^*$ in degree 4.

\section{4-dimensional case}\label{sect4dimCase}

In this section, we explicitly describe the kernel of
$\bar\iota^*$ in degree $4$ when $n=2$, that is, when $M$ is of
dimension $4$. In this case, $\partial(M/T)$ is the union of
$b_1+1$ closed polygonal curves.

First we recall the case when $b_1=0$.  In this case,
$H^{even}_T(M)=H^*_T(M)$.  Let  $\partial(M/T)$ be an $m$-gon and
$v_1,\dots,v_m$ be the primitive edge vectors in the multi-fan of
$M$, where $v_i$ and $v_{i+1}$ span a 2-dimensional cone for every
$i=1,2,\dots,m$ (see \cite{ma-pa13}).  Note that $v_i\in H_2(BT)$
and we understand $v_{m+1}=v_1$ and $v_0=v_m$ in this section.
Since $\{v_j, v_{j+1}\}$ is a basis of $H_2(BT)$ for every $j$, we have
$\det(v_j,v_{j+1})=\pm 1$.

Let $\tau_i\in H^2_T(M)$ be the equivariant Poincar\'e dual class to the
characteristic submanifold corresponding to $v_i$.  Then we have
\begin{equation} \label{eq:5-5}
\pi^*(u)=\sum_{i=1}^m\langle u,v_i\rangle\tau_i\quad \text{for every $u\in H^2(BT)$},
\end{equation}
where $\langle\ ,\ \rangle$ denotes the natural pairing between
cohomology and homology, (see \cite{masu99} for example). We
multiply both sides in \eqref{eq:5-5} by $\tau_i$.  Then, since
$\tau_i\tau_j=0$ if $v_i$ and $v_j$ do not span a 2-dimensional
cone, \eqref{eq:5-5} turns into
\begin{equation} \label{eq:5-7}
0=\langle u,v_{i-1}\rangle \tau_{i-1}\tau_i+\langle
u,v_i\rangle\tau_i^2+\langle u,v_{i+1}\rangle\tau_i\tau_{i+1}
\quad\text{in $H^*_T(M)/(\pi^*(H^2(BT)))$.}
\end{equation}
If we take $u$ with $\langle u,v_i\rangle=1$, then \eqref{eq:5-7}
shows that $\tau_i^2$ can be expressed as a linear combination of
$\tau_{i-1}\tau_i$ and $\tau_{i}\tau_{i+1}$.  If we take
$u=\det(v_i,\ )$, then $u$ can be regarded as an element of $H^2(BT)$
because $H^2(BT)=\Hom(H_2(BT),\Z)$. Hence, \eqref{eq:5-7} reduces
to
\begin{equation} \label{eq:5-8}
\det(v_{i-1},v_i)\tau_{i-1}\tau_i=\det(v_i,v_{i+1})\tau_i\tau_{i+1}
\quad\text{in $H^*_T(M)/(\pi^*(H^2(BT)))$.}
\end{equation}
Finally we note that $\tau_i\tau_{i+1}$ maps to the cofundamental
class of $M$ up to sign. We denote by $\mu\in H^4_T(M)$ the
element (either $\tau_{i-1}\tau_i$ or $-\tau_{i-1}\tau_i$) which
maps to the cofundamental class of $M$.

When $b_1\ge 1$, the above argument works for each component of
$\partial(M/T)$.  In fact, according to \cite{masu99},
\eqref{eq:5-5} holds in $H^*_T(M)$ modulo $H^*(BT)$-torsion but in
our case there is no $H^*(BT)$-torsion in $H^{even}_T(M)$ by
Proposition~\ref{prop:1-2}. Suppose that $\partial(M/T)$ consists
of $m_j$-gons for $j=1,2,\dots,b_1+1$.  To each $m_j$-gon, we have
the class $\mu_j\in H^4_T(M)$ (mentioned above as $\mu$).  Since
$\mu_j$ maps to the cofundamental class of $M$,  $\mu_i-\mu_j$
$(i\not=j)$ maps to zero in $H^4(M)$; so it is in the kernel of
$\bar{\iota}^*$.  The subgroup of $H^{even}_T(M)/(\pi^*(H^2(BT)))$
in degree 4 generated by $\mu_i-\mu_j$ $(i\not=j)$ has the desired
rank $b_1$.

%

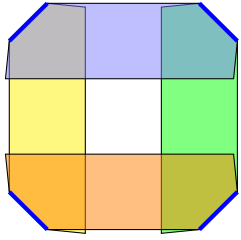
\begin{figure}[h]
    \begin{center}
        \begin{tikzpicture}[scale=.5]
            \pgfsetfillopacity{0.5}
            \filldraw[fill=yellow] (1,0)--(2,-0.1)--(2,5.9)--(1,6)--(0,5)--(0,1)--cycle;
            \filldraw[fill=green] (4,-0.1)--(5,0)--(6,1)--(6,5)--(5,6)--(4,5.9)--cycle;
            \filldraw[fill=orange] (-0.1,2)--(0,1)--(1,0)--(5,0)--(6,1)--(5.9,2)--cycle;
            \filldraw[fill=blue!50] (-0.1,4)--(5.9,4)--(6,5)--(5,6)--(1,6)--(0,5)--cycle;
            \draw[ultra thick,blue] (1,0)--(0,1);
            \draw[ultra thick,blue] (5,0)--(6,1);
            \draw[ultra thick,blue] (6,5)--(5,6);
            \draw[ultra thick,blue] (0,5)--(1,6);
        \end{tikzpicture}
    \end{center}
    \caption{The origami template with four polygons}
\label{fig:cycle}
\end{figure}

\begin{example}
Take the 4-dimensional toric origami manifold $M$ corresponding to
the origami template shown on Figure \ref{fig:cycle} (Example 3.15
of \cite{ca-gu-pi11}). Topologically $M/T$ is homeomorphic to
$S^1\times [0,1]$ and the boundary of $M/T$ as a manifold with
corners consists of two closed polygonal curves, each having $4$
segments. The multi-fan of $M$ is the union of two copies of the
fan of $\C P^1\times \C P^1$ with the product torus action.
Indeed, if $v_1$ and $v_2$ are primitive edge vectors in the fan of $\C
P^1\times \C P^1$ which spans a 2-dimensional cone, then the other
primitive edge vectors $v_3,\dots,v_8$ in the multi-fan of $M$ are
\[
v_3=-v_1,\quad v_4=-v_2,\quad \text{and}\quad v_{i}=v_{i-4} \quad
\text{for $i=5,\dots,8$}
\]
and the 2-dimensional cones in the multi-fan are
\[
\begin{split}
&\angle v_1v_2,\quad \angle v_2v_3,\quad \angle v_3v_4,\quad \angle v_4v_1,\\
&\angle v_5v_6,\quad \angle v_6v_7,\quad \angle v_7v_8,\quad \angle v_8v_5,
\end{split}
\]
where $\angle vv'$ denotes the 2-dimensional cone spanned by
vectors $v$ and $v'$.  Note that
\begin{equation} \label{eq:5-9-0}
\text{$\tau_i\tau_j=0$ if $v_i, v_j$ do not span a 2-dimensional cone.}
\end{equation}

We have
\begin{equation} \label{eq:5-9}
\pi^*(u)=\sum_{i=1}^8\langle u,v_i\rangle \tau_i \quad \text{for every $u\in H^2(BT)$.}
\end{equation}
Let $\{v_1^*,v_2^*\}$ be the dual basis of $\{v_1,v_2\}$.  Taking
$u=v_1^*$ or $v_2^*$ , we see that
\begin{equation} \label{eq:5-9-1}
\tau_1+\tau_5=\tau_3+\tau_7,\quad \tau_2+\tau_6=\tau_4+\tau_8\quad\text{in $H^*_T(M)/(\pi^*(H^2(BT)))$}.
\end{equation}
Since we applied \eqref{eq:5-9} to the basis $\{v_1^*, v_2^*\}$ of
$H^2(BT)$, there is no other essentially new linear relation among
$\tau_i$'s.

Now, multiply the equations \eqref{eq:5-9-1} by $\tau_i$ and use
\eqref{eq:5-9-0}.  Then we obtain
\[
\begin{split}
&\tau_i^2=0\quad \text{for every $i$},\\
&(\mu_1:=)\tau_1\tau_2=\tau_2\tau_3=\tau_3\tau_4=\tau_4\tau_1,\\
&(\mu_2:=) \tau_5\tau_6=\tau_6\tau_7=\tau_7\tau_8=\tau_8\tau_5\quad\text{in $H^*_T(M)/(\pi^*(H^2(BT)))$}.
\end{split}
\]
Our argument shows that these together with \eqref{eq:5-9-0} are
the only degree two relations among $\tau_i$'s {in
$H^*_T(M)/(\pi^*(H^2(BT)))$}. The kernel of
\[
\bar\iota^*\colon H^{even}_T(M;\Q)/(\pi^*(H^{2}(BT;\Q)))\to
H^{even}(M;\Q)
\]
in degree 4 is spanned by $\mu_1-\mu_2$.
\end{example}

\section{On the cokernel of $\bar\iota^*$ in degree $2$} \label{sectCokernel}

In this section we describe the elements of $H^2(M)$ that do not
lie in the image of the map \eqref{eq:5-4-1}. In fact, we describe
geometrically the homology $(2n-2)$-cycles, which are Poincar\'e
dual classes to the elements in $H^2(M)$. A very similar technique was used in
\cite{po-sa11} to calculate the homology of 4-dimensional torus
manifolds whose orbit spaces are polygons with holes. In contrast
to \cite{po-sa11} we do not introduce particular cell structures
on $M$, because this approach becomes more complicated for higher
dimensions.

Denote the orbit space $M/T$ by $Q$, so $Q$ is a manifold with
corners and acyclic proper faces, and $Q$ is homotopy equivalent
to a wedge of $b_1$ circles. Also let $q\colon M\to Q$ denote the
projection to the orbit space, and $\Gamma_i$ be the
characteristic subgroup, i.e. the stabilizer of orbits in
$F_i^{\circ}\subset Q$. For each face $G$ of $Q$, we denote by
$\Gamma_G$ the stabilizer subgroup of orbits $x\in G^{\circ}$.
Thus $\Gamma_G=\prod_i\Gamma_i\subset T$, where the product is
taken over all $i$ such that $G\subseteq F_i$. The origami
manifold $M$ is homeomorphic to the model
$ Q\times T/\sim,$
where $(x_1,t_1)\sim(x_2,t_2)$ if $x_1=x_2\in G^{\circ}$ and
$t_1t_2^{-1}\in \Gamma_G$ for some face $G\subset Q$. This fact is
a consequence of a general result of the work \cite{yo11}. In the
following we identify $M$ with the model $Q\times T/\sim$.

Consider a homology cycle $\sigma\in H_{n-1}(Q,\partial Q)$. Note
that $\sigma$ is Poincare--Lefschetz dual to some element of
$H^1(Q)\cong H^1(\bigvee_{b_1}S^1)\cong \Z^{b_1}$. Let $\sigma$ be
represented by a pseudomanifold $\xi\colon(L,\partial L)\to
(Q,\partial Q)$, where $\dim L=n-1$, and let $[L]\in
H_{n-1}(L,\partial L)$ denote the fundamental cycle, so that
$\xi_*([L])=\sigma$. We assume that $\xi(L\setminus
\partial L)\subset Q\setminus\partial Q$. Moreover, since every face of
$\partial Q$ is acyclic, we may assume that $\xi(\partial L)$ is
contained in $\partial Q^{(n-2)}$, --- the codimension $2$
skeleton of $Q$. A pseudomanifold $(L,\partial L)$ defines a
collection of $(2n-2)$-cycles in homology of $M$, one for each
codimension-one subtorus of $T$, by the following construction.

\begin{con}\label{conNonEquiCycle}
First fix a coordinate splitting of the torus,
$T=\prod_{i\in[n]}T_i^1$ in which the orientation of each $T_i^1$
is arbitrary but fixed. For each $j\in [n]$ consider the subtorus
$T_{\jh}=T^{[n]\setminus j}=\prod_{i\in[n]\setminus j}T_i^1$, and
let $\inc\colon T_{\jh}\to T$ be the inclusion map. Given a
pseudomanifold $(L,\partial L)$ as in the previous paragraph,
consider the space $L\times T_{\jh}$ and the quotient construction
$(L\times T_{\jh})/\sim_*$, where the identification $\sim_*$ is
naturally induced from $\sim$ by the map $\xi$. Since
$\xi(\partial L)\subset \partial Q^{(n-2)}$, the space $(\partial
L\times T_{\jh})/\sim_*$ has dimension at most $2n-4$. Thus
$(L\times T_{\jh})/\sim_*$ has the fundamental cycle $V_{L,j}$.
Indeed, there is a diagram:
\begin{equation}\label{eqFundCycleOfCut}
\xymatrix@C-=0.26cm{ &&H_{n-1}\left(L,\partial L\right)\otimes H_{n-1}(T_{\jh}) \ar@{->}[d]^{\cong \mbox{ (Kunneth)}} &\\
0\ar@{=}[d]&& H_{2n-2}(L\times T_{\jh},\partial L\times T_{\jh})
\ar@{->}[d]^{\cong \mbox{
(excision)}}&0\ar@{=}[d]\\
H_{2n-2}(\frac{\partial L\times T_{\jh}}{\sim_*}) \ar@{->}[r]&
H_{2n-2}(\frac{L\times T_{\jh}}{\sim_*})\ar@{->}[r]^(0.41){\cong}&
H_{2n-2}(\frac{L\times T_{\jh}}{\sim_*}, \frac{\partial L\times
T_{\jh}}{\sim_*}) \ar@{->}[r]& H_{2n-3}(\frac{\partial L\times
T_{\jh}}{\sim_*}) }
\end{equation}
Let $T_{\jh}$ be oriented by the splitting $T\cong T_{\jh}\times
T_j^1$. Given such an orientation, there exists the distinguished
generator $\Omega_j\in H_{n-1}(T_{\jh})$. Then the fundamental
cycle $V_{L,j}\in H_{2n-2}((L\times T_{\jh})/\sim_*)$ is defined
as the image of $[L]\otimes \Omega_j\in H_{n-1}(L,\partial
L)\otimes H_{n-1}(T_{\jh})$ under the isomorphisms of diagram
\eqref{eqFundCycleOfCut}. The induced map
\[
\zeta_{L,j}\colon (L\times T_{\jh})/\sim_* \to (Q\times
T_{\jh})/\sim \hookrightarrow (Q\times T)/\sim=M
\]
determines the element $x_{L,j}=(\zeta_{L,j})_*(V_{L,j})\in
H_{2n-2}(M)$.
\end{con}

\begin{proposition}
Let $\{\sigma_1,\ldots,\sigma_{b_1}\}$ be a basis of
$H_{n-1}(Q,\partial Q)$, and let $L_1$, \dots, $L_{b_1}$ be
pseudomanifolds representing these cycles and satisfying the
restrictions stated above. Consider the set of homology classes
$\{x_{L,j}\}\subset H_{2n-2}(M)$, where $L$ runs over the set
$\{L_1,\ldots,L_{b_1}\}$ and $j$ runs over $[n]$. Then the set of
Poincare dual classes of $x_{L,j}$ is a basis of the cokernel
$H^2(M)/\iota^*(H^2_T(M))$.
\end{proposition}

\begin{proof}
Consider disjoint circles $S_1,\dots,S_{b_1}\subset Q^{\circ}$
whose corresponding homology classes $[S_1],\dots,[S_{b_1}]\in
H_1(Q)$ are dual to $\sigma_1,\ldots,\sigma_{b_1}$. Thus for the
intersection numbers we have $[S_i]\cap \sigma_k=\delta_{ik}$.
Consider 2-dimensional submanifolds of the form $S_i\times
T^1_{l}\subset M$, where $l\in [n]$. They lie in
$q^{-1}(Q^{\circ})\subset M$. Let $[S_i\times T^1_{l}]\in H_2(M)$
be the homology classes represented by these submanifolds. Then
\begin{equation}\label{eq:pairing}
[S_i\times T^1_{l}]\cap x_{L_k,j} = \delta_{ik}\delta_{lj},
\end{equation}
since all the intersections lie in $q^{-1}(Q^{\circ}) =
Q^{\circ}\times T$.

The equivariant cycles of $M$ sit in $q^{-1}(\partial Q)$. Thus
the intersection of $S_i\times T^1_{l}\subset q^{-1}(Q^{\circ})$
with any equivariant cycle is empty. Nondegenerate pairing
\eqref{eq:pairing} shows that the set $\{x_{L,j}\}$ is linearly
independent modulo equivariant cycles. Its cardinality is
precisely $nb_1$ and the statement follows from
Lemma~\ref{lemm:2-1-1}.
\end{proof}

\begin{remark} The element $x_{L,j}\in H_{2n-2}(M)$ depends on the
representing pseudomanifold $L$, not only on its homology class in
$H_{n-1}(Q,\partial Q)$. The classes corresponding to different
representing pseudomanifolds are connected by linear relations
involving characteristic submanifolds. We describe these relations
next.
\end{remark}

At first let us introduce orientations on the objects under
consideration. We fix an orientation of the orbit space $Q$. This
defines an orientation of each facet ($F_i$ is oriented by
$TF_i\oplus \nu\cong TQ$, where the inward normal vector of the
normal bundle $\nu$ is set to be positive). Since the torus $T$ is
oriented, we have a distinguished orientation of $M = Q\times
T/\sim$. Recall that $\Gamma_i$ is the characteristic subgroup
corresponding to a facet $F_i\subset Q$. Since the action is
locally standard, $\Gamma_i$ is a $1$-dimensional connected
subgroup of $T$. Let us fix orientations of all characteristic
subgroups (this choice of orientations is usually called an
omniorientation). Then every $\Gamma_i$ can be written as
\begin{equation}\label{eq:charsubgrp}
\Gamma_i=\{(t^{\lambda_{i,1}},\ldots,t^{\lambda_{i,n}})\in T\mid
t\in T^1\},
\end{equation}
where $(\lambda_{i,1},\ldots,\lambda_{i,n})\in \Z^n$ is a uniquely
determined primitive integral vector.

Let us orient every quotient torus $T/\Gamma_i$ by the following
construction. For each $\Gamma_i$ choose a codimension 1 subtorus
$\Upsilon_i\subset T$ such that the product map $\Upsilon_i\times
\Gamma_i\to T$ is an isomorphism. The orientations of $T$ and
$\Gamma_i$ induce an orientation of $\Upsilon_i$. The quotient map
$T\to T/\Gamma_i$ induces an isomorphism between $\Upsilon_i$ and
$T/\Gamma_i$ providing the quotient group with an orientation. The
orientation of $T/\Gamma_i$ defined this way does not depend on
the choice of the auxiliary subgroup $\Upsilon_i$.

Finally, the orientations on $F_i$ and $T/\Gamma_i$ give an
orientation of the characteristic submanifold $M_i=q^{-1}(F)$.
This follows from the fact that $M_i$ contains an open dense
subset $q^{-1}(F_i^{\circ}) = F_i^{\circ}\times(T/\Gamma_i)$.

\begin{con}
Let $F_i$ be a facet of $Q$, and $[F_i]\in H_{n-1}(F_i,\partial
F_i)$ its fundamental cycle. The cycles $[F_i]$ form a basis of
$$H_{n-1}(\partial Q,\partial Q^{(n-2)})=\bigoplus_{\mbox{facets}}
H_{n-1}(F_i,\partial F_i).$$
Let $\xi_\varepsilon\colon
(L_\varepsilon,\partial L_\varepsilon)\to (Q,\partial Q)$,
$\varepsilon=1,2$, be two pseudomanifolds representing the same
element $\sigma\in H_{n-1}(Q,\partial Q)$. Then there exists a
pseudomanifold $(N,\partial N)$ of dimension $n$ and a map
$\eta\colon N\to Q$ such that $L_1$ and $L_2$ are disjoint submanifolds
of $\partial N$, $\eta|_{L_\varepsilon}=\xi_\varepsilon$ for
$\varepsilon=1,2$, and $\eta(\partial N\setminus
(L_1^{\circ}\sqcup L_2^{\circ}))\subset
\partial Q$ (this follows from the geometrical definition of homology,
see \cite[App. A.2]{ru-sa}). The skeletal stratification of $Q$
induces a stratification on $N$. The restriction of the map $\eta$
sends $\partial N^{(n-2)}$ to $\partial Q^{(n-2)}$. Let $\delta$
be the connecting homomorphism
\[
\delta\colon H_n(N,\partial N)\to H_{n-1}(\partial
N,\partial N^{(n-2)})
\]
in the long exact sequence of the triple $(N,\partial N,\partial
N^{(n-2)})$. Consider the sequence of homomorphisms
\begin{multline*}
H_n(N,\partial N)\xrightarrow{\delta} H_{n-1}(\partial N,\partial
N^{(n-2)}) \cong \\ H_{n-1}(L_1,\partial L_1)\oplus
H_{n-1}(L_2,\partial L_2)\oplus H_{n-1}(\partial N\setminus
(L_1^{\circ}\cup L_2^{\circ}), \partial N^{(n-2)}) \\ \xrightarrow{\id\oplus\id\oplus \eta_*}
H_{n-1}(L_1,\partial L_1)\oplus H_{n-1}(L_2,\partial L_2)\oplus
H_{n-1}(\partial Q,\partial Q^{(n-2)}).
\end{multline*}
This sequence of homomorphisms sends the fundamental cycle $[N]\in
H_n(N,\partial N)$ to the element
\begin{equation}\label{eqRelPseudBordism}
\left([L_1],-[L_2],\sum\nolimits_i\alpha_i[F_i]\right)
\end{equation}
of the group $H_{n-1}(L_1,\partial L_1)\oplus H_{n-1}(L_2,\partial
L_2)\oplus H_{n-1}(\partial Q,\partial Q^{(n-2)})$, for some
coefficients $\alpha_i\in \Z$.
\end{con}

\begin{proposition}\label{propLinearRelation}
If $L_1$ and $L_2$ are two pseudomanifolds representing a class
$\sigma\in H_{n-1}(Q,\partial Q)$, then for each $j\in [n]$
\begin{equation}\label{eqZeroCombination}
x_{L_1,j}-x_{L_2,j} +
\sum_{\mbox{facets}}\alpha_i\lambda_{i,j}[M_i]=0\quad\mbox{ in
}H_{2n-2}(M).
\end{equation}
Here $M_i$ is the characteristic submanifold of $M$ corresponding
to $F_i$, the numbers $\alpha_i$ are given by
\eqref{eqRelPseudBordism}, and the numbers $\lambda_{i,j}$ are
given by \eqref{eq:charsubgrp}.
\end{proposition}

\begin{proof}
Choose a relative pseudomanifold bordism $N$ between $L_1$ and
$L_2$ and consider the space $(N\times T_{\jh})/\sim_*$. Here
$\sim_*$ is the equivalence relation induced from $\sim$ by the
map $\eta$. We have a map $(\eta\times\inc)/\sim\colon(N\times
T_{\jh})/\sim_*\to M$. By the diagram chase, similar to
\eqref{eqFundCycleOfCut}, the space $(N\times T_{\jh})/\sim_*$ is
a $(2n-1)$-pseudomanifold with boundary. Its boundary represents
the element \eqref{eqZeroCombination}. Thus this element vanishes
in homology. We only need to prove the following technical lemma.

\begin{lemma}
Let $F_i$ be a facet, and let $\Gamma_i$ be its characteristic subgroup
encoded by the vector $(\lambda_{i,1},\ldots,\lambda_{i,1})\in
\Z^n$ for $j\in[n]$. Let $\Omega_j\in H_{n-1}(T_{\jh})$ and
$\Phi_i\in H_{n-1}(T/\Gamma_i)$ be the fundamental classes (in the
orientations introduced previously). Then the composite map
$T_{\jh}\hookrightarrow T\twoheadrightarrow T/\Gamma_i$ sends
$\Omega_j$ to $\lambda_{i,j}\Phi_i$.
\end{lemma}

\begin{proof}
Let $\{\mathbf{e}_s\mid s\in[n]\}$ be the positive basis of
$H_1(T)$ corresponding to the splitting $T=\prod_sT_s^1$, and let
$\{\mathbf{f}_r\mid r\in[n]\}$ be a positive basis of $H_1(T)$
such that $\mathbf{f}_n=(\lambda_{i,1},\ldots,\lambda_{i,n})$.
Thus $\Omega_j=(-1)^{n-j}\mathbf{e}_1\wedge \ldots\wedge
\widehat{\mathbf{e}_j}\wedge\ldots\wedge \mathbf{e}_n$. Let $D$ be
the matrix of basis change,
$\mathbf{e}_s=\sum_{r=1}^nD_s^r\mathbf{f}_r$, and let $C=D^{-1}$.
The element $\Omega_j$ maps to
\[
(-1)^{n-j}\det(D_s^r)_{\begin{subarray}{l} r\in \{1,\ldots,n-1\}
\\ s\in \{1,\ldots,\hat{j},\ldots,n\}\end{subarray}}
\]
which is equal to the element $C_n^j$ by Cramer's rule. $C_n^j$ is
the $j$th coordinate of $\mathbf{f}_n$ in the basis
$\{\mathbf{e}_s\}$. Thus, by construction, $C_n^j =
\lambda_{i,j}$.
\end{proof}

This proves the proposition.
\end{proof}

Proposition \ref{propLinearRelation} gives the idea how to
describe the multiplication in $H^*(M)$. Equivalently, we need to
describe the intersections of cycles in $H_*(M)$. Intersections of
equivariant cycles are known --- they are encoded by the face ring
of $Q$. To describe the intersections of additional cycles
$x_{L,j}$ sometimes we can do the following:

(1) Let $M_F$ be the face submanifold of $M$, corresponding to the
face $F\subset~Q$. If $F\cap \partial L=\varnothing$, then
$[M_F]\cap x_{L,j}=0$ in the homology of $M$. Otherwise, in many
cases we can choose a different representative $L'$ of the same
homology class as $L$ with the property $\partial L'\cap F=0$.
Then, by Proposition \ref{propLinearRelation},
\[\begin{split}
[M_F]\cap
x_{L,j}&=[M_F]\cap x_{L',j}+[M_F]\cap \sum_{\mbox{facets}}
\alpha_i\lambda_{i,j}[M_i]\\
&=\sum_{\mbox{facets}}
\alpha_i\lambda_{i,j}[M_F]\cap [M_i]
\end{split}
\]
 which can be computed using
relations in $\Bbbk[Q]/(\theta_1,\dots,\theta_n)$.

(2) To compute the intersection of two elements of the form
$x_{L_1,j_1}$ and $x_{L_2,j_2}$ sometimes we can use the same
trick: find a pseudomanifold $L_1'$ which does not intersect $L_2$
and replace $x_{L_1,j_1}$ by $x_{L'_1,j_1}+\sum_i
\alpha_i\lambda_{i,j_1}[M_i]$. Then the intersection
$x_{L'_1,j_1}\cap x_{L_2,j_2}$ vanishes and intersections of
$x_{L_2,j_2}$ with $[M_i]$ are computed using~(1).


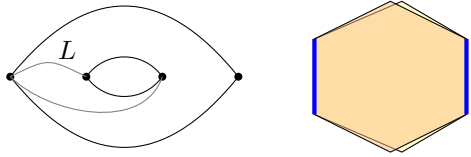
\begin{figure}[h]
\begin{center}
    \begin{tikzpicture}[scale=.5]
        \draw  (0,0)..controls (0.5,0.7) and (1.5,0.7)..(2,0);
        \draw  (0,0)..controls (0.5,-0.7) and (1.5,-0.7)..(2,0);
        \draw (-2,0)..controls (0,2.5) and (2,2.5)..(4,0);
        \draw (-2,0)..controls (0,-2.5) and (2,-2.5)..(4,0);
        \fill (-2,0) circle(3pt);
        \fill (0,0) circle(3pt);
        \fill (2,0) circle(3pt);
        \fill (4,0) circle(3pt);
        \draw[gray] (-2,0)..controls (-1,.5)..(0,0);
        \draw[gray] (-2,0)..controls (-1,-1) and (1.5,-1.5)..(2,0);
        \draw (-0.5,0.7) node{$L$};
        \pgfsetfillopacity{0.5}
        \filldraw[fill=yellow!50] (6,-1)--(6,1)--(8.3,2)--(10,1)--(10,-1)--(8.3,-2)--cycle;
        \filldraw[fill=orange!50] (6,-1)--(6,1)--(8,2)--(10,1)--(10,-1)--(8,-2)--cycle;
        \draw[ultra thick, blue] (6,-1)--(6,1);
        \draw[ultra thick, blue] (10,-1)--(10,1);
    \end{tikzpicture}
\end{center}
\caption{Manifold with corners $Q$ for which the products of extra
elements cannot be calculated using linear relations of
Proposition~\ref{propLinearRelation}} \label{fig:2gons}
\end{figure}

\begin{remark}
This general idea may not work in particular cases. Figure
\ref{fig:2gons} provides an example of $Q$ such that every
pseudomanifold $L$ with $\partial L\subset
\partial Q^{(0)}$, representing the generator of $H_1(Q,\partial Q)$,
intersects every facet of $Q$. Unfortunately, such situations may
appear as realizations of origami templates. The picture on the
right shows an origami template, whose geometric realization is
the manifold with corners shown on the left.
\end{remark}

\section{Some observation on non-acyclic
cases}\label{sectNonAcyclicFaces}

The face acyclicity condition we assumed so far is not preserved
under taking the product with a symplectic toric manifold $N$, but
every face of codimension $\ge \frac{1}{2}\dim N+1$ is acyclic.
Motivated by this observation, we will make the following
assumption on our toric origami manifold $M$ of dimension $2n$:
\begin{quote}
every face of $M/T$ of codimension $\ge r$ is acyclic for some integer~$r$.
\end{quote}
Note that $r=1$ in the previous sections. Under the above
assumption, the arguments in Section~\ref{sectBettiNumbers} work
to some extent in a straightforward way.  The main point is that
Lemma~\ref{lemm:3-5} can be generalized as follows.

\begin{lemma} \label{lemm:7-1}
The homomorphism $H^{2j}(\tM)\to H^{2j}(Z_+\cup Z_-)$ induced from
the inclusion is surjective for $j\ge r$.
\end{lemma}

Using this lemma, we see that Lemma~\ref{lemm:3-6} turns into the
following.

\begin{lemma} \label{lemm:7-2} We have the relations
\[
\begin{split}
&\sum_{i=1}^r(b_{2i}(\tilde M)-b_{2i-1}(\tilde M))=\sum_{i=1}^r(b_{2i}(M)-b_{2i-1}(M))+b_{2r}(B)\\
&b_{2i}(\tilde M)-b_{2i-1}(\tilde M)=b_{2i}(M)-b_{2i-1}(M)+b_{2i}(B)-b_{2i-2}(B)\quad\text{for $i\ge r+1$}.
\end{split}
\]
\end{lemma}

Combining Lemma~\ref{lemm:7-2} with Lemma~\ref{lemm:3-4}, Lemma~\ref{lemm:3-7} turns into the following.

\begin{lemma} \label{lemm:7-3} We have the relations
\[
\begin{split}
&b_1(M')=b_1(M)-1,\quad b_{2r}(M')=b_{2r}(M)+b_{2r-2}(B)+b_{2r}(B),\\
&b_{2i+1}(M')=b_{2i+1}(M)\quad \text{for $r\le i\le n-r-1$}.
\end{split}
\]
\end{lemma}

Finally, Theorem~\ref{theo:3-1} is generalized as follows.

\begin{theorem}
Let $M$ be an orientable toric origami manifold of dimension $2n$
$(n\ge 2)$ such that every face of $M/T$ of codimension $\ge r$ is
acyclic.  Then
\[
b_{2i+1}(M)=0\quad\text{for $r\le i\le n-r-1$}.
\]
Moreover, if $M'$ and $B$ are as above, then
\[
\begin{split}
&b_1(M')=b_1(M)-1\,\,(\text{hence $b_{2n-1}(M')=b_{2n-1}(M)-1$ by Poincar\'e duality}),\\
&b_{2i}(M')=b_{2i}(M)+b_{2i}(B)+b_{2i-2}(B)\quad\text{for $r\le i\le n-r$}.
\end{split}
\]
\end{theorem}

\end{document}